\documentclass[11pt,reqno]{amsart}

\usepackage[utf8]{inputenc}
\usepackage{fullpage}
\usepackage{bbm}
\usepackage[all]{xy}
\usepackage{color}
\usepackage{url}
\usepackage{hyperref}
\usepackage{cleveref}
\usepackage{mathtools}
\usepackage{amssymb}

\newcommand{\ann}{\mathrm{ann}}

\usepackage{enumitem}

\newcommand{\bbGamma}{{\mathpalette\makebbGamma\relax}}
\newcommand{\makebbGamma}[2]{%
\raisebox{\depth}{\scalebox{1}[-1]{$\mathsurround=0pt#1\mathbb{L}$}}%
}

\newcommand{\Mdef}[2]{\newcommand{#1}{\relax \ifmmode #2 \else $#2$\fi}}
\Mdef{\bT}{\mathbb{T}}
\newcommand{\sfT}{\mathsf{T}}
\newcommand{\sfD}{\mathsf{D}}

\newcommand{\tensor}{\otimes}
\newcommand{\unit} {\mathbbm{1}}
\newcommand{\unitadalg}{\bS^{alg}_{\ad}}
\newcommand{\cC}{\mathcal{C}}
\newcommand{\D}{\mathcal{D}}
\newcommand{\cG}{\mathcal{G}}
\newcommand{\cW}{\mathcal{W}}
\newcommand{\fp}{\mathfrak{p}}
\newcommand{\fq}{\mathfrak{q}}
\newcommand{\gen}{\mathfrak{g}}
\newcommand{\fm}{\mathfrak{m}}
\newcommand{\fn}{\mathfrak{n}}
\newcommand{\modules}{\text{-mod}}

\newcommand{\cO}{\mathcal{O}}
\newcommand{\cK}{\mathcal{K}}

\newcommand{\bP}{\mathbb{P}}
\newcommand{\bL}{\mathbb{L}}
\newcommand{\T}{\mathbb{T}}
\newcommand{\holim}{\mathop{ \mathop{\mathrm {holim}} \limits_\leftarrow} \nolimits}
\newcommand{\bS}{\mathbb{S}}
\newcommand{\cF}{\mathcal{F}}

\newcommand{\Hhat}{\hat{H}}
\newcommand{\bZ}{\mathbb{Z}}
\newcommand{\Z}{\mathbb{Z}}
\newcommand{\Q}{\mathbb{Q}}
\newcommand{\bQ}{\mathbb{Q}}

\newcommand{\Qc}{\textbf{QCoh}}
\newcommand{\cA}{\mathcal{A}}
\newcommand{\cE}{\mathcal{E}}

\newcommand{\cCad}{\cC_\mathrm{ad}}
\newcommand{\cCsep}{\cC_\mathrm{sep}}

\newcommand{\cCLk}{\cC_\mathrm{L\kappa}}
\newcommand{\cCLkqc}{\cC_\mathrm{L\kappa}^{\qc}}

\newcommand{\cCadwqce}{\cC_\mathrm{ad}^{\wqce}}

\newcommand{\cCadqce}{\cC_\mathrm{ad}^{\qce}}
\newcommand{\cCadqc}{\cC_\mathrm{ad}^{\qc}}
\newcommand{\cCade}{\cC_\mathrm{ad}^{e}}

\newcommand{\cCsepqcie}{\cC_\mathrm{sep}^{\qcie}}
\newcommand{\cCsepqc}{\cC_\mathrm{sep}^{\qc}}

\newcommand{\cGad}{\cG_\mathrm{ad}}

\newcommand{\Lfm}{L_0^{\fm}}

\newcommand{\RmmodLfmhat}{\Lfm\mbox{-$\Rm \modules$}}

\newcommand{\lra}{\longrightarrow}
\newcommand{\lla}{\longleftarrow}

\newcommand{\e}{\mathrm{e}}
\newcommand{\qcie}{\mathrm{qcie}}
\newcommand{\ie}{\mathrm{ie}}
\newcommand{\qce}{\mathrm{qce}}
\newcommand{\qc}{\mathrm{qc}}
\newcommand{\we}{\mathrm{we}}
\newcommand{\wqc}{\mathrm{wqc}}
\newcommand{\wqcie}{\mathrm{wqcie}}
\newcommand{\wqce}{\mathrm{wqce}}
\newcommand{\wie}{\mathrm{wie}}
\newcommand{\wqck}{\mathrm{wqc\kappa}}

\newcommand{\cCb}{\operatorname{Ho}(\cC)}
\newcommand{\supp}{\mathrm{supp}}
\newcommand{\Hom}{\mathrm{Hom}}
\newcommand{\spc}{\mathrm{Spc}}
\newcommand{\spco}{\mathrm{Spc}^{\omega}}
\newcommand{\spec}{\mathrm{Spec}}
\newcommand{\Gie}{\Gamma_{\ie}}
\newcommand{\Gqc}{\Gamma_{\qc}}
\newcommand{\Gqce}{\Gamma_{\qce}}
\newcommand{\Ge}{\Gamma_{e}}
\newcommand{\cCsepie}{\cCsep^{ie}}
\newcommand{\sigmab}{\overline{\sigma}}

\newcommand{\cCal}{\cC_{\alpha}}
\newcommand{\cCFal}{\cC_{F_*\alpha}}
\newcommand{\cCFalqc}{\cC_{F_*\alpha}^{qc}}
\newcommand{\cCalqc}{\cC_{\alpha}^{qc}}
\newcommand{\cN}{\mathcal{N}}
\newcommand{\cV}{\mathcal{V}}

\newcommand{\cQ}{\mathcal{Q}}
\newcommand{\etahat}{\hat{\eta}}
\newcommand{\epshat}{\hat{\epsilon}}
\newcommand{\sep}{\mathrm{sep}}
\newcommand{\Lk}{\mathrm{L\kappa}}

\newcommand{\st}{\; | \;}

\newcommand{\pb}{p}
\newcommand{\aaa}{a}
\newcommand{\ad}{\mathrm{ad}}
\newcommand{\cell}{R}
\newcommand{\cS}{\mathcal{S}}
\newcommand{\qqed}{\qed \\[1ex]}

\newcommand{\cCm}{\cC_\fm}
\newcommand{\Rm}{R_\fm}

\newcommand{\fibre}{\mathrm{fibre}}

\newtheorem{thm}{Theorem}[section]
\newtheorem*{thm*}{Theorem}
\newtheorem{lemma}[thm]{Lemma}
\newtheorem{prop}[thm]{Proposition}
\newtheorem{cor}[thm]{Corollary}

\theoremstyle{definition}
\newtheorem{defn}[thm]{Definition}
\newtheorem{defn-prop}[thm]{Definition--Proposition}

\newtheorem{example}[thm]{Example}

\newtheorem{remark}[thm]{Remark}

\DeclareRobustCommand{\SkipTocEntry}[5]{}

\begin{document}
\title{Separated and complete adelic models for one-dimensional Noetherian tensor-triangulated categories}

\author[Balchin]{Scott Balchin}
\address[Balchin]{Max Planck Institute for Mathematics, Vivatsgasse 7,
  53111 Bonn, Germany}
\email{balchin@mpim-bonn.mpg.de}
\author[Greenlees]{J.P.C.Greenlees}
\address[Greenlees]{Warwick Mathematics Institute, Zeeman
  Building, Coventry, CV4 7AL, UK}
\email{john.greenlees@warwick.ac.uk}
 \date{\today}

\begin{abstract}
We prove the existence of various adelic-style models for rigidly small-generated
tensor-triangulated categories whose Balmer spectrum is a
one-dimensional Noetherian topological space. This special case of our
general programme of giving adelic models is particularly concrete and
accessible, and we  illustrate it with examples from algebra, geometry, topology and representation theory.
\end{abstract}

\maketitle

\tableofcontents

\section{Introduction}

In this paper we are concerned with triangulated categories $\sfT$ equipped with a compatible symmetric
monoidal structure $(\tensor, \unit)$ forming a 
tensor-triangulated category in the sense of \cite{BalmerSpc}. Our
particular focus  is on those tensor-triangulated categories which
have the formal properties of the derived category of a 1-dimensional
commutative Noetherian domain or an irreducible curve. The
irreducibility assumption is inessential, and made only to remove
distractions, whereas the Noetherian assumption is essential for the
simplicity of our approach. Many interesting examples are covered by this hypothesis, some of which are discussed in \Cref{sec:examples}.

\subsection{Adelic and algebraic models}
Structural features of tensor-triangulated categories are controlled
by the Balmer spectrum~\cite{BalmerSpc}, which is the counterpart of
the Zariski prime spectrum of a commutative ring. Given a prime $\fp$ in the Balmer spectrum
of $\sfT$, one may define the localization at~$\fp$ (denoted $L_{\fp}$)
and  the completion at $\fp$ (denoted $\Lambda_{\fp}$). These
constructions are analogues of those in commutative algebra.

The idea of decompositions reflecting the dimensional
filtration of primes is pervasive in commutative
algebra and algebraic geometry. The aim  is to show that the subquotients in this 
filtration can be understood in terms of primes of a single
dimension. There are decisions to be made about the way the pure
strata are described, and how the assembly information is packaged. 
Since we are restricting ourselves to  the 1-dimensional case in this
paper there is
only one step to the assembly process, often described in terms of recollements.
Discussions of this type can be found in~\cite{stratified,Barthel,HA},
and we plan to return to the higher dimensional case elsewhere. 

The adelic models project \cite{BalchinGreenlees} 
also follows this general pattern, but has some 
distinctive features. The word `adelic' arises because it aims to construct models
for $\mathsf{T}$ which are built from localized-completed objects,
with the assembly information packaged in adele-like localizations of 
products.  The second 
distinctive element is the idea of trying to make the building blocks
and assembly information as algebraic as possible. In the most
favourable circumstances the building blocks are (derived) categories of modules
over a commutative ring and the assembly data is described in terms of 
a diagram of rings. 

This project arose from the  earlier work of the second
author and collaborators on algebraic models for rational
$G$-spectra, starting with the circle group~\cite{Greenlees99, GreenleesShipley18,
  BarnesGreenleesKedziorekShipley17}. In that case,  the adelic model could be 
shown to be formal, hence showing both that the homotopy category is
determined by the pieces it is built from and also that the result is
algebraic.

It is therefore natural that there are two elements to the adelic
project. The first breaks down a tensor-triangulated category in an adelic fashion, and
the second  studies algebraic examples of this type. 
The present paper follows this pattern, by firstly establishing that
adelic models exist in considerable generality, and secondly showing
that in an algebraic context these models may be further simplified.  
It is in the nature of the goals that our principal examples will be close to
algebra (modules over commutative rings, algebraic geometry,
representation theory, {\em rational} $G$-spectra). Although our homotopical 
models apply generally,  the extent to which
they add to alternative approaches will depend on 
calculational specifics of individual cases. 

This paper focuses on the case when the  Balmer spectrum is a
Noetherian space, where the topology is determined by the Balmer
spectrum  as a partially ordered set. In the 1-dimensional case 
this just means the space of closed points has the cofinite
topology. Many naturally occurring examples are not Noetherian:
ongoing work of the authors with Barthel aims towards constructing models
in this more general case via the yoga of spectral spaces~\cite{prismatic}.

\subsection{Objectwise decomposition}

The one-dimensional irreducible case of~\cite[Theorem 8.1]{BalchinGreenlees} tells
us that under some mild hypothesis on $\sfT$ the unit object can be
recovered by the  homotopy pullback of ring objects
\begin{equation}\label{adelicsquare}
\begin{gathered} \xymatrix{
\unit \ar[r] \ar[d]  & L_{\gen} \unit \ar[d] \\
\displaystyle{\prod_{\fm}} \Lambda_{\fm} \unit \ar[r] & L_{\gen} \displaystyle{\prod_{\fm}} \Lambda_{\fm} \unit \rlap{ .}
} \end{gathered}
\end{equation} 
Here $\gen$ is the generic point of the Balmer spectrum and $\fm$ ranges over the closed points (i.e., Balmer minimal primes).

When $\sfT = \sfD(\bZ)$, with unit the integers $\bZ$, this recovers the
classical Hasse square. If $\sfT$ is the homotopy
category of rational $\bT$-spectra (where $\bT$ is the circle group)
with unit the sphere spectrum $\bS$ then this is the Tate square for the family of
 finite subgroups~\cite{GreenleesMay95}.

The idea is to use this fracturing of the monoidal unit to provide
small and computationally advantageous models of $\sfT$ while also
giving an insight into the global structure of the category. The first
model related to this fracturing is the \emph{adelic model}
of~\cite{BalchinGreenlees}. This model reconstructs an object $X \in \mathsf{T}$ from modules over the rings $L_{\gen}\unit, \prod_{\fm}\Lambda_{\fm}\unit$ and $L_{\gen}\prod_{\fm}\Lambda_{\fm}\unit$. In essence, it reconstructs an object $X \in \sfT$ using the homotopy pullback square
\begin{equation}\label{fracture1}
\begin{gathered} \xymatrix{
X \ar[r] \ar[d]  & L_{\gen} \unit \tensor X \ar[d] \\
\left[ \displaystyle{\prod_{\fm}} \Lambda_{\fm} \unit \right]\tensor X \ar[r] &
\left[ L_{\gen} \displaystyle{\prod_{\fm}} \Lambda_{\fm} \unit \right] \tensor X \rlap{ .}
} \end{gathered}
\end{equation}
obtained as the tensor product of Diagram~\ref{adelicsquare} with $X$.

This result goes a long way towards reconstructing $\sfT$ out of
categories of modules over localized-complete rings. However, we note
that modules over the product ring $\prod_\fm \Lambda_\fm \unit$ can
be complicated. The first aim of this paper is to replace this single
module category by the separate module categories over the simpler
rings $\Lambda_\fm \unit$, and we refer to this as the  \emph{separated model}. In essence, it reconstructs an object $X \in \sfT$ using the homotopy pullback square
\begin{equation}\label{fracture2}
\begin{gathered} \xymatrix{
X \ar[r] \ar[d]  & L_{\gen} \unit \otimes X \ar[d] \\
\displaystyle{\prod_{\fm}} \left[ \Lambda_{\fm} \unit \otimes X\right] \ar[r] &
L_{\gen} \displaystyle{\prod_{\fm}} \left[\Lambda_{\fm} \unit \otimes X \right] \rlap{ .}
} \end{gathered}
\end{equation}

Secondly, one might hope to reconstruct $X$ from its actual completions,
which are often more economical, and we will also 
construct a \emph{complete  model}  of this type. In essence, it reconstructs an object
$X \in \sfT$ using the homotopy pullback square
\begin{equation}\label{fracture3}
\begin{gathered} \xymatrix{
X \ar[r] \ar[d]  & L_{\gen} \unit \tensor X 
\ar[d] \\
\displaystyle{\prod_{\fm}} \left[ \Lambda_{\fm}  X\right] \ar[r] &
L_{\gen} \displaystyle{\prod_{\fm}} \left[\Lambda_{\fm}  X \right] \rlap{ .}
} \end{gathered}
\end{equation}

Putting it all together, we have a ladder in which all the horizontal
fibres are equivalent. This means all subrectangles are homotopy
pullback squares, and in particular the three involving $X$ at the top
left are the three discussed above.
$$\xymatrix{
X \ar[r] \ar[d] &L_{\gen}\unit \tensor X\ar[d]\\
(\prod_\fm \Lambda_{\fm}\unit)\tensor X \ar[r] \ar[d] &L_{\gen}(\prod_\fm \Lambda_{\fm}\unit)\tensor X\ar[d] \\
\prod_\fm (\Lambda_{\fm}\unit\tensor X) \ar[r] \ar[d] &L_{\gen}\prod_\fm (\Lambda_{\fm}\unit\tensor X)\ar[d] \\
\prod_\fm (\Lambda_{\fm} X) \ar[r]  &L_{\gen}\prod_\fm (\Lambda_{\fm} X) 
}$$

\begin{remark}
By taking horizontal fibres, one can rebuild arbitrary objects from local and torsion data. Promoting this to a categorical decomposition is undertaken in~\cite{adelict}, which retrieves the torsion model of the second author from~\cite{Greenlees99}. As all of the horizontal fibres coincide, we see that up to homotopy there is only one notion of a torsion model, and the only difference occurs in which category we wish to splice over.
\end{remark}

\subsection{Homotopical and algebraic models}
\label{subsec:htpcalalg}
The point of the paper is to build models based on these objectwise
decompositions. We describe these in detail in Section \ref{sec:framework},
but sketch it here. They all take the form of a diagram
\begin{equation}\label{modulediagram}
\begin{gathered} \xymatrix{
 & V \ar[d] \\
N\ar[r] & Q
} \end{gathered}
\end{equation}
 of modules over the diagram of rings 
\begin{equation}\label{adeliccospan}
\begin{gathered} \xymatrix{
 & L_{\gen} \unit \ar[d] \\
\displaystyle{\prod_{\fm}} \Lambda_{\fm} \unit \ar[r] & L_{\gen} \displaystyle{\prod_{\fm}} \Lambda_{\fm} \unit \rlap{ .}
} \end{gathered}
\end{equation}
obtained from Diagram \ref{adelicsquare} by omitting the unit
object. The
homotopical models simply consist of suitable model structures on the category
of diagrams. The adelic model requires cofibrant objects to have the
horizontal and vertical base change maps weak
equivalences. Accordingly, following \cite{Greenlees99}, we think of
the object as the object $N$
(the {\em nub}) together with the additional structure of $V$ (the
{\em vertex}) and the splicing embodied by maps to $Q$. The separated
model requires in addition that the object $N$ is weakly equivalent to the product of its
idempotent pieces, and slightly weakens the condition on the vertical
base change map. The complete model requires that
$N$ is equivalent to a product and each idempotent piece is
homotopically complete,
and there is no condition on the vertical base change map. The existence of these homotopical models appears as Theorem~\ref{thm:adsepcomp}.

In an algebraic context we can go further in each of these three
cases. In each case (adelic, separated or
complete) the condition on cofibrant objects requires certain base
change maps to
be weak equivalences. We may consider the subcategory  of objects  in
which these maps are not just equivalences but isomorphisms. Under
suitable conditions this category   is well behaved and admits a model
structure by restriction: the subcategory
is said to be a  \emph{skeleton} and is Quillen equivalent to the
homotopical model. The algebraic models are the content of Corollary~\ref{cor:adcellskel} and Theorems~\ref{thm:sepskeleton} and \ref{thm:compskeleton}.

Altogether then we have three homotopical models (adelic, separated
and complete), and in well behaved algebraic contexts each of these
has a  skeleton, giving six models.

 \subsection{Outline of paper}
   
In \Cref{sec:framework} we recall the apparatus of
completions, localizations and diagram categories
from~\cite{BalchinGreenlees}, and recall the general Hasse pullback square for
the unit object.  

The remainder of the paper is divided into three
parts. In \Cref{part1} (Sections \ref{sec:adelicm} and \ref{sec:adelics}), we focus on models which apply rather generally:
they are homotopical in the sense that the important subcategories are
identified by homotopical conditions. That is, we explore the theory
hinted at in \ref{subsec:htpcalalg} above.

In \Cref{part2} (Sections \ref{sec:roadmap} to
\ref{sec:skeletonequiv}), we focus on examples of an algebraic nature
where we can identify much smaller subcategories specified by strict conditions, providing skeletons. 

Finally, in \Cref{sec:examples} (Sections \ref{sec:abgrp} to \ref{sec:Tspectra}) 
we explicitly illustrate the theory for commutative rings,
quasi-coherent sheaves over a curve, the stable module category for the
Klein 4-group and rational circle-equivariant spectra.

\subsection*{Acknowledgements}
The authors were supported by the grant EP/P031080/2. They are
grateful to T.~Barthel, L.~Pol and J.~Williamson for related
discussions of the adelic framework over an extended period,
and to D.~J.~Benson for discussions about the stable module  category
of the  Klein 4-group. They thank the referee for detailed comments. 
The first author would like to thank the Max Plank Institute for  Mathematics for its
hospitality. 
 
\section{The adelic framework}\label{sec:framework}

In this section we briefly recall the relevant machinery from \cite{BalchinGreenlees} that we will use throughout.

\subsection{Noetherian model categories}

We will work in the setting of Quillen model categories.  We are
interested in the homotopy categories of model categories $\cC$ which
are both stable and monoidal.  We continue with the conventions of
\cite{adelic} where the word `compact' refers to the model category
and `small' to the homotopy category.  We assume the
homotopy category $\cCb$ is a rigidly small-generated
tensor-triangulated category~\cite{HoveyMC,HPS}.  
We write  $\cCb^\omega$ for the full subcategory of small
objects, our assumptions guarantee that these object coincide with the dualizable objects.

 We denote by $\mathcal{G} \subseteq \operatorname{Ho}(\cC)^\omega$ a set of compact objects which generate the category $\cCb$. In particular an object $X \in \operatorname{Ho}(\cC)$ is zero if and only if
$[g, X]_* = 0 \text{ for every } g \in \mathcal{G}.$ 

In this context we may introduce the main organizational principle from~\cite{BalmerSpc}.

\begin{defn}
A {\em prime ideal} in a tensor-triangulated category is a 
proper thick tensor ideal $\fp$ with the
property that $a\tensor b\in \fp$ implies that $a$ or $b$ is in $\fp$. The {\em Balmer spectrum} of a tensor-triangulated category $\cCb $ is
the set of prime tensor ideals of the triangulated category of compact objects: 
$$\spco (\cCb)=\{ \fp \subseteq \cCb^\omega \st \fp \mbox{  is prime} \}.$$ 
\end{defn}

We  will restrict attention to  Balmer spectra that are  Noetherian in the sense that
chains of open sets satisfy the ascending chain condition. Noetherian spectral spaces enjoy a simple description.

\begin{lemma}\label{lem:noethersp}
Let $X$ be a Noetherian spectral space. The closed subsets of $X$ are precisely the finite unions of the closures of points. In particular, the topology on $X$ is determined by the poset structure given by the specialization order. 
\end{lemma}

As in \cite[\S 4]{BalchinGreenlees}, a {\em Noetherian model category} is a cellular, proper, stable,
monoidal model category $\cC$ such that $\cCb$ is a rigidly
small-generated tensor-triangulated category whose Balmer spectrum is
a Noetherian topological space.

We say that a  prime $\fp$ is \emph{visible} if  there is a small
object $K_\fp$ so that $\overline{\{\fp\}} =
\supp(K_{\fp}) = \{\fq \in \spco(\cCb) \mid K_\fp \not\in \fq \}$. In
light of Lemma~\ref{lem:noethersp}, we have the following useful characterisation of Noetherian spectral spaces.

\begin{lemma}[{\cite[Corollary 2.17]{BalmerSpc}}]\label{lem:noeth}
The topological space $\spco(\cCb)$ is Noetherian if and only if all primes are visible.
\end{lemma}


\subsection{Localizations and completions}
Localization and completion are central to our approach, and we
summarize material from~\cite[\S5,6]{BalchinGreenlees}. We freely use
results about left and right Bousfield localization from~\cite{Hirschhorn03}.

Let $\fp$ be an arbitrary Balmer
prime. The localization at $\fp$, denoted $L_{\fp}$, is
the nullification of the thick tensor subcategory $\fp$. By assumption,
the objects of $\fp$ are all small, and as such $L_\fp$ is a finite
localization, and therefore is smashing in the sense that there is a weak equivalence $L_\fp X \simeq L_\fp \unit \tensor X$. 

For a Balmer prime $\fp$, we choose a  small object
$K_\fp$ such that $\supp(K_\fp) = \overline{\{\fp\}}$, which exists by
the assumption that the Balmer spectrum is
Noetherian. Any two generate the same thick tensor ideal so subsequent
constructions do not depend on the choice. The completion at $\fp$, denoted $\Lambda_{\fp}$ is then
the  Bousfield localization with respect to the object $K_\fp$. In
particular we shall say that a map $f \colon X \to Y$ is a weak
equivalence in $\Lambda_{\fp} \cC$ if $f \tensor K_\fp \colon X
\tensor K_\fp \to Y \tensor K_\fp$ is a weak equivalence in
$\cC$. Unlike the $L_{\fp}$, these completions are not usually smashing,
and do not usually preserve small objects.
Both of these Bousfield localizations exist for Noetherian model
categories and are stable monoidal localizations~\cite[\S
5,6]{BalchinGreenlees}.

We also need to consider cellularization (right Bousfield
localization)  with respect to the thick tensor ideal generated by $K_{\fp}$, giving a
$K_{\fp}\tensor \cG $-equivalence $\Gamma_{\fp}X\lra X$
from a $K_{\fp}\tensor \cG$-cellular object.  This gives a useful equivalence $\Lambda_{\fp}X\simeq \Hom
(\Gamma_{\fp}\unit, X)$. 

For the rest of the article we will  assume in addition that the Balmer
spectrum is 1-dimensional and irreducible. This means it has a single
 generic point $\gen$ (the Balmer-maximal
prime) and a set of closed points $\fm$  (Balmer-minimal primes). 



\begin{remark}
The geometry of completions and localizations suggests highlighting
the downward cone for completion and cellularization  ($\Lambda_\fp =
\Lambda_{\wedge(\fp)}$ and $\Gamma_\fp = \Gamma_{\wedge(\fp)}$) and
the upward cone for localization ($L_\fp =L_{\vee (\fp)}$).  For brevity
we will retain the simpler traditional notation. 
\end{remark}


The main result from~\cite{BalchinGreenlees}  tells us that we can
retrieve the monoidal unit $\unit$ from its localized completions. All
of the models appearing in this paper will come from various
interpretations of this result. 

\begin{thm}[Adelic Approximation Theorem]
\label{thm:AAT}
Let $\cC$ be a 1-dimensional Noetherian model category whose Balmer spectrum is one-dimensional and irreducible. Then the unit object $\unit$ is the homotopy pullback of the following diagram of rings in $\cC$
$$\unit_{ad} := \left[ \begin{gathered} \xymatrix{
& L_{\gen} \unit \ar[d]^-j \\
\displaystyle{\prod_{\fm}} \Lambda_{\fm} \unit \ar[r]_-{L_{\gen}} & L_{\gen} \displaystyle{\prod_{\fm}} \Lambda_{\fm} \unit
} \end{gathered}\right]$$
\end{thm}

\begin{proof}
One takes the horizontal fibre and sees that the map $\Gamma_\gen
\unit \to \Gamma_\gen \prod_\fm \Lambda_\fm \unit$ is a weak
equivalence since $\unit \to \prod_\fm \Lambda_\fm \unit$ is a
$K_{\fm}$ equivalence for all $\fm$. This is the one-dimensional version of the general result~\cite[Theorem 
8.1]{BalchinGreenlees} and a special case of the Tate square
\cite[Corollary 2.4]{Greenlees01b}. 
\end{proof}

\subsection{Diagrams of model categories}
\label{subsec:qce}
Diagrams of model categories are at the core of our method. The only diagram shape that we will need to consider is that of a cospan:
$$\alpha= \left( \begin{gathered}\xymatrix{
& \cC_1 \ar[d]^{F_1} \\ \cC_2 \ar[r]_{F_2} & \cC_3 
} \end{gathered}\right)  
$$
A generalized diagram over this cospan is displayed as
$$ \begin{gathered} \xymatrix{
& c_1 \ar@{..>}[d]^{f_1} \\ c_2 \ar@{..>}[r]_{f_2} & c_3 . 
} \end{gathered} 
$$
More explicitly, it is a quintuple 
$$(c_1,c_2,c_3,f_1 \colon F_1 c_1 \to c_3, f_2 \colon F_2c_2 \to c_3)$$
where $c_i \in \cC_i$. A morphism
$$(c_1,c_2,c_3,f_1 \colon F_1 c_1 \to c_3, f_2 \colon F_2c_2 \to c_3) \to (c_1',c_2',c_3',f_1' \colon F_1 c_1' \to c_3', f_2' \colon F_2c_2' \to c_3')$$
is given by morphisms $\alpha_i \colon c_i \to c_i'$ in $\cC_i$ for $i=0,1,2$ such that the squares in the following diagram commute:
\[
\xymatrix{
F_1 c_1  \ar[d]_{F_1 \alpha_1} \ar[r]^{f_1} & c_3 \ar[d]|{\alpha_3} & \ar[l]_{f_2} \ar[d]^{F_2 \alpha_2} F_2 c_2 \\
F_1 c_1'  \ar[r]_{f_1'} & c_3' & \ar[l]^{f_2'} F_2 c_2' 
}
\]

These  diagrams and morphisms between them form a category: we think of the categories
$\cC_i$ as fibres over the points of the diagram shape so that the 
(generalized) diagram is a {\em section}, and write 
$$\cCal = \bbGamma \left( \begin{gathered}\xymatrix{
& \cC_1 \ar[d]^{F_1} \\ \cC_2 \ar[r]_{F_2} & \cC_3 
} \end{gathered} \right)
$$
for the category of sections.

If we further assume that the categories $\cC_i$ come equipped with
model structures, and that the functors $F_1,F_2$ are left Quillen functors
then the category of sections admits the injective model
structure in which weak equivalences and cofibrations are objectwise ~\cite[Theorem
3.1]{GreenleesShipley14b}. The injective model structure
inherits many properties that all the $\cC_i$ have, such as being
proper and cellular. We note that this  injective model structure is 
the lax homotopy limit of the diagram in the sense of \cite{Bergner12}.

\subsection{Quasi-coherence and extendedness}\label{sec:qce}

In our situation, the functor $F_2$ appearing in the diagram category
will have the character of a localization, and  the vertical functor
$F_1$ will have the flavour of a faithfully flat extension.

\begin{defn}
Let $X = (c_1,c_2,c_3,f_1 \colon F_1 c_1 \to c_3, f_2 \colon F_2c_2
\to c_3)$  be an object of $\cCal$.
We say that $X$ is:
\begin{itemize}
    \item \emph{quasi-coherent ($\qc$)} 
 if the map $F_2c_2 \to c_3$ is an isomorphism  in $\cC_3$.
    \item \emph{extended ($\e$)}  if the map $F_1c_1 \to c_3$ is an isomorphism in $\cC_3$.
    \item if we only require weak equivalences in the injective model
      structure, we say that $X$  is \emph{weakly quasi-coherent
        ($\wqc$)} or \emph{weakly extended ($\we$)}. We write $\wqce$
      for the class of
      weakly quasi-coherent, weakly extended objects.
\end{itemize}
\end{defn}

We find it helpful to write 

$$
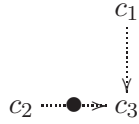
\begin{gathered} \xymatrix{
& c_1 \ar@{..>}[d] \\ c_2 \ar@{..>}[r]|-{\text{{\Large $\bullet$}}} & c_3 
} \end{gathered} $$

to indicate a quasi-coherent object and 

$$
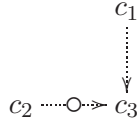
\begin{gathered} \xymatrix{
& c_1 \ar@{..>}[d] \\ c_2 \ar@{..>}[r]|-{\text{{\Large $\circ$}}} & c_3 
} \end{gathered} $$

to indicate a weakly quasi-coherent object. Similar notation applies
for (weakly) extended and (weakly) quasi-coherent extended objects, as
well as for subcategories of such objects.

\subsection{Homotopical models}\label{intro:homotopical}
Having described the diagrammatic framework, we can outline how to lift
the decomposition on objects  to a categorical level. Full details
will be given in Sections \ref{sec:adelicm} and \ref{sec:adelics}. 

We assume that $\sfT$ is the homotopy category of some stable monoidal model category $\cC$
with 1-dimensional, irreducible, Noetherian Balmer spectrum, and
consider the diagram of rings in Display \ref{adeliccospan}. The
adelic category $\cCad$ consists of modules over this diagram of
rings. 
Our homotopical models will be formed from $\cCad$ by Bousfield
localization. For a model category  $\D$, we write $R_A\D$ for the cellularization with
respect to the thick tensor ideal generated by the set $A$ of objects of $\D$, and for an object $E$ of
$\D$ we write $L_E\D$ for the left  localization inverting $E$-homology isomorphisms. 

One of the main results of~\cite{BalchinGreenlees} states that  $\cC$ is Quillen equivalent to a model in which the
bifibrant objects are the weakly quasi-coherent extended ($\wqce$)
objects in $\cCad$. These are exactly the objects which are reconstructed as in Diagram~\ref{fracture1}.
More precisely,  we give $\cCad$ the diagram
injective model structure and then form the cellularization $R_\wqce
\cCad$  and obtain an equivalence $R_\wqce \cCad \simeq_Q \cC$.

To access the objects arising in the situation of
Diagram~\ref{fracture2} we begin instead by  describing a slightly
larger class $\wqcie \supseteq \wqce$ of objects that treats the different closed 
  points separately.  If we use the
  less extreme localization $R_\wqcie \cCad$ we do not get a model of
  $\cC$ since too many objects remain bifibrant.  To compensate for  this we can apply a mild  left localization $L_\Pi$  to ensure  that the fibrant objects have the object $M$ equivalent to a product of modules 
  over the factors $\Lambda_{\fm}\unit$. This leads to a
  model $L_\Pi R_\wqcie \cC \simeq_Q \cC$ which we call the {\em
    separated  model}. It is homotopical in the sense  that the ambient category still has all $\prod_\fm \Lambda_\fm
  \unit$-modules at the nub, and they are only required to be products up to
  homotopy.  

Less extreme still, we can form the even  larger class $\wqc \supseteq 
  \wqcie\supseteq \wqce$ of weakly quasi-coherent objects. We then form the
  cellularization 
$R_\wqc \cCad$ with 
  respect to $\wqc$. Once again we do not yet arrive at a model of $\cC$. However, we perform the same balancing act as above, and  apply a more extreme left 
  localization $L_\Lambda$ to obtain
  one. This localization is taken in such a way to force the fibrant objects to have the module $M$ to have the homotopy type not only a product of modules, but the product of complete modules. That is, such a model is the reconstruction with respect to Diagram~\ref{fracture3}. As such we call $L_\Lambda R_\wqc \cC \simeq_Q \cC$ the \emph{complete model}.
  
  To summarize, we will prove that there are Quillen equivalences
  $$\cC\simeq_Q R_\wqce \cCad \simeq_Q L_\Pi R_\wqcie \cCad \simeq_Q L_\Lambda
R_\wqc \cCad$$
 which provides us with three models whose underlying category is $\cCad$.

\subsection{Algebraic models}\label{intro:alg}
    When the categories in question are of an
  algebraic nature,  we may consider the  cellular skeleton, where the
  underlying category is pared  down to the absolute minimum so that
  the homotopy theory has the least amount of work to do. For example,
  we consider the full subcategory $\cCadqce \subseteq \cCad$ of
  objects whose base change maps are isomorphisms (as opposed to
  being weak equivalences). One can then prove that the restricted
  model structure on $\cCadqce$ induced from the diagram-injective
  model on $\cCad$ is Quillen equivalent to $R_\wqce \cCad$ and hence
  also to $\cC$. 
  
Similarly, the  skeleton of the separated model requires that
the nub is actually a product, and a suitably adapted version of the
vertical basechange map is an isomorphism.  For a  skeleton of
the complete model, we need the more stringent restriction to an 
 algebraic context where there is an abelian model for complete
 modules, but in that case we again obtain a skeleton. 
  
\subsection{Summary}

To summarize, in the most favourable situation we obtain three
different models. In this setting, an object $X
\in \operatorname{Ho}(\cC)$ can be described by each of the following
three sets of  data:
\begin{itemize}[align=left]
    \item[Adelic:]\leavevmode
        \begin{itemize}
            \item $V$ a $L_\gen \unit$-module;
            \item $N$ a  $\prod_\fm \Lambda_\fm \unit$-module;
            \item An equivalence $L_\gen N \simeq j_\ast V$ of  $L_\gen \prod_\fm \Lambda_\fm \unit$-modules.
        \end{itemize}
    \item[Separated:]\leavevmode
            \begin{itemize}
            \item $V$ a $L_\gen \unit$-module;
            \item $N_\fm$ an  $\Lambda_\fm \unit$-module for each $\fm$;
            \item A map of $ j_\ast V\lra L_\gen \prod_\fm N_\fm$ of  $L_\gen
              \prod_\fm \Lambda_\fm \unit$-modules so that for each
              idempotent $e_\fm$ it gives an equivalence  $e_\fm j_\ast V
              \xrightarrow{\simeq} e_\fm L_\gen N_\fm$.
        \end{itemize}
    \item[Complete:]\leavevmode
                \begin{itemize}
            \item $V$ a  $L_\gen \unit$-module;
            \item $N_\fm$ a complete $\Lambda_\fm \unit$-module for each $\fm$;
              \item A map $V \to L_\gen \prod_\fm N_\fm $ of  $L_\gen \prod_\fm \Lambda_\fm \unit$-modules.
        \end{itemize}
\end{itemize}
In \Cref{sec:abgrp} we make these explicit in the case of the derived category of a 1-dimensional
Noetherian domain, and in Section \ref{sec:qcsheaves} for quasi-coherent sheaves over a curve.

\part{Homotopical models}\label{part1}

In Part 1 we describe the homotopical adelic, separated and complete models, which
apply very generally in the 1-dimensional, Noetherian setting. 

\section{The adelic model}\label{sec:adelicm}

In this section we recall the adelic model
from~\cite{BalchinGreenlees} and prove that  this cellularization of the adelic
category is a realisation of  the strict homotopy limit.

\subsection{The adelic model as a cellularization}

We now have the necessary language to describe the adelic model.
We will freely use the technology of module categories in monoidal model categories from~\cite{SchwedeShipley00}. For $R$ a cofibrant monoid in a monoidal model category $\cC$ we denote by $R \text{-mod}_\cC$ the associated category of modules which we equip with the projective model structure.

\begin{defn}
Let $\cC$ be a 1-dimensional Noetherian model category. We 
define the adelic  category to be the category of 
sections of the adelic cospan: 
$$
\cCad :=  \bbGamma \left( \begin{gathered}\xymatrix{
& (L_\gen \unit) \modules_{\cC} \ar[d]^{j_\ast} \\ \left(  \prod_\fm \Lambda_\fm \unit \right) \modules_{\cC} \ar[r]_-{L_{\gen}} & \left(L_\gen  \prod_\fm \Lambda_\fm \unit \right)\modules_{\cC}
}  \end{gathered} \right)
$$
The functors in this diagram are extensions of scalars, and therefore
left adjoints. In other words, we take the category of modules over the diagram of rings $\unit_{ad}$ from Theorem~\ref{thm:AAT}.
\end{defn}

The categories are related by an adjoint pair 
\[
\xymatrix{\cC \ar@<.75ex>[r]^{\aaa} \ar@<-.75ex>@{<-}[r]_{\pb}  \ar@{}[r] & \cCad}
\]
where $\aaa$ is defined via pointwise tensoring with the diagram $\unit_{ad}$, and $\pb$ is the pullback of the diagram when viewed in $\cC$.

We equip the individual module categories with the module projective
model structures, and $\cCad$ with the diagram injective model
structure.  With these model structures the above adjoint pair is a Quillen pair. One checks that on homotopy categories
$\aaa$ preserves small objects, and that the derived unit is a weak
equivalence as a consequence of the 
Adelic Approximation Theorem \ref{thm:AAT} which tells us that $\pb(a(\unit)) = \pb(\unit_{ad})\simeq \unit$. 

Recall that we have assumed that $\cCb$ has a set $\cG$ of small generators. 
We shall write $\cGad$ for the image of the
generators under $\aaa$. We now apply the cellularization principle \cite{GreenleesShipley13}
 to obtain the basic adelic model. 

\begin{thm}\cite[Theorem 9.3]{BalchinGreenlees}
\label{equiv1}
Let $\cC$ be a 1-dimensional Noetherian model category. The above
adjunction is a Quillen equivalence
$$ \cC \simeq_Q \cell_{\cGad} \cCad$$
between $\cC$ and the cellularization of $\cC_{ad}$ at the image of the generators.
\end{thm}


\subsection{The adelic model as a strict homotopy limit}\label{subsec:cellstrictholim}

The cellularization appearing in Theorem~\ref{equiv1} has an
attractive universal property: it coincides with the strict
homotopy limit as defined by Bergner and
Barwick~\cite{Barwick10,Bergner12} which we will denote
$\mathcal{L}\mathrm{im}(\cCad)$ (\cite[9.3]{BalchinGreenlees})).  
We will show that the subcategory of bifibrant objects in 
$\cell_{\cGad} \cCad$ consists of $\wqce$-objects which are
object-wise cofibrant. We can describe this subcategory as
$$
\cCad^\wqce = \bbGamma \left( \begin{gathered}\xymatrix{
& (L_\gen \unit) \modules_\cC \ar[d]|{\text{\Large $\circ$}} \\ \left(  \prod_\fm \Lambda_\fm \unit \right) \modules_\cC \ar[r]|-{\text{\Large $\circ$}} & \left(L_\gen  \prod_\fm \Lambda_\fm \unit \right) \modules_\cC
}  \end{gathered} \right)
$$
where we have used the $\circ$ notation as in \Cref{sec:qce}, and we
have supressed notation to indicate that we are requiring diagrams
to be object-wise cofibrant. 

We now introduce some special objects of $\cCad^{\wqce}$ by considering the
right adjoints to evaluation at a prime. First we consider the generic prime $\gen$. 
 For an arbitrary $L_{\gen} \unit$-module $W$, define the object 
 $$e(W) = \left[ \begin{gathered} \xymatrix{
& W \ar@{..>}[d]|-{\text{\Large $\circ$}} \\ j_\ast W \ar@{..>}[r]|-{\text{\Large $\circ$}} & j_\ast W 
} \end{gathered} \right] \in \cCad^{\wqce}.$$
In the bottom left hand corner notation for restriction of scalars along the ring map $\prod_\fm \unit \to L_\gen \prod \Lambda_\fm \unit$ has been omitted.

An $L_{\gen}$-module $W$ can be viewed  as an object of $\cC$ via the
forgetful functor to which we can apply the functor $a$, giving 
a functor from $L_{\gen}$-modules to $\cCad^{\wqce}$.

\begin{lemma}
The object $e(W)$ is equivalent to an object coming from $\cC$ in that
it is in the image of $a$. Indeed, $e(W)$ is 
the image of $W$ itself viewed as a $\unit$-module in $\cC$. The functor $e \colon (L_{\gen}\unit) \modules_\cC \to \cCad$ is right adjoint to evaluation at $\gen$ (i.e., at the vertex) on homotopy category of $\cCad$. 
\end{lemma}

\begin{proof}
By definition
$$
a(W) =\left[  \begin{gathered} \xymatrix{
& L_{\gen} \unit \otimes W \ar@{..>}[d]|-{\text{\Large $\circ$}} \\
\prod_{\fm}\Lambda_{\fm} \unit \tensor W  
\ar@{..>}[r]|-{\text{\Large $\circ$}} &  L_{\gen}\prod_{\fm}\Lambda_{\fm} \unit \tensor W  
} \end{gathered} \right]
$$
Since $W$ is a $L_\gen\unit$-module we have $L_{\gen} \unit \otimes W \simeq W$, and by definition
the right hand vertical functor is extension of scalars. Since the bottom
left entry is $L_{\gen}$-local, and the lower horizontal functor is
$L_{\gen}$-localization it follows that the bottom left entry is also
equivalent to $j_*W$. Hence $a(W)\simeq e(W)$ as required.

For the adjoint property we see 
$$\left[ 
\left[  \begin{gathered} \xymatrix{
& V \ar@{..>}[d]|-{\text{\Large $\circ$}} \\ N \ar@{..>}[r]|-{\text{\Large $\circ$}} & P 
} \end{gathered} \right]
 ,  \left[ \begin{gathered} \xymatrix{
& W \ar@{..>}[d]|-{\text{\Large $\circ$}} \\ j_\ast W \ar@{..>}[r]|-{\text{\Large $\circ$}} & j_\ast W 
} \end{gathered} \right] 
\right] = [V,W].$$
\end{proof}

When it comes to minimal primes, there is no right adjoint on the
whole subcategory. However, if  $N$ is a $ \left(  \prod_\fm \Lambda_\fm
  \unit \right)$-module which is torsion in the sense  that $L_{\gen}N
\simeq 0$, then we may take 
 $$f(N) = \left[ \begin{gathered} \xymatrix{
& 0 \ar@{..>}[d]|-{\text{\Large $\circ$}} \\ N \ar@{..>}[r]|-{\text{\Large $\circ$}} & 0
} \end{gathered} \right] \in \cCad^{\wqce}$$

\begin{lemma}
When $L_{\gen}N\simeq 0$, the object $f(N)$ is equivalent to an object
in the image of $a$. Indeed, $f(N)$ is 
the image of $N$ itself viewed as a $\unit$-module.  The object $f(-)$ has the property
 $\Hom(f(M),f(N)) = \Hom(M,N)$ and $[f(M), f(N)]=[M,N]$.
\end{lemma}

\begin{proof}
We may calculate
$$
a(N) =\left[  \begin{gathered} \xymatrix{
& L_{\gen} N \ar@{..>}[d]|-{\text{\Large $\circ$}} \\
\prod_{\fm}\Lambda_{\fm}\unit\tensor N
\ar@{..>}[r]|-{\text{\Large $\circ$}} &  j_*L_{\gen}N 
} \end{gathered} \right]
$$
Since $L_{\gen}N\simeq 0$, we have that $a(N)\simeq
f(\prod_{\fm}\Lambda_{\fm}\unit\tensor N)$. As $\prod_{\fm}\Lambda_{\fm}\unit\tensor N \simeq N$ the result follows as the object is weakly extended. 


The statement about maps between $f(M)$ and $f(N)$ is clear due to the fact that the diagrams are concentrated at a single vertex. 
\end{proof}

\begin{remark}\label{rem:supported}
These $f(-)$ and $e(-)$ objects are very useful since they correspond
to two layers given by the dimension filtration. That is, the objects $e(-)$ are ones
supported on the prime $\gen$, while the $f(-)$ are supported on the
closed points $\fm$. In particular, these objects  build
all $\wqce$ objects. Indeed, for an arbitrary object
$$
X = \left[ \begin{gathered} \xymatrix{
& V \ar@{..>}[d]|-{\text{\Large $\circ$}} \\ N \ar@{..>}[r]|-{\text{\Large $\circ$}} & P
} \end{gathered} \right] \in \cCad^{\wqce}
$$
there is a natural map $X\lra e(V)$ corresponding to the identity on
$V$ and this gives a  cofibre sequence
\begin{align}
\label{cof:fe}
f(N') \to X \to e(V), 
\end{align}
where $N'=\fibre(N\lra L_{\gen}N)$. 
\end{remark}

\begin{thm}\label{equiv2} \cite[9.C]{BalchinGreenlees}
Every  object of $\cell_{\cGad} \cCad$ is equivalent to an
object of $\cCad^{\wqce}$ and hence 
$$\cC\simeq_Q \cell_{\cGad} \cCad\simeq_Q \mathcal{L}\mathrm{im} (\cCad).$$
\end{thm}

\begin{proof}
By \cite[3.2]{Bergner12} the strict homotopy limit is the
cellularization of the diagram category with respect to
$\wqce$. We have right Bousfield localized with respect to  the apparently smaller subcategory of objects coming from $\cC$, so it suffices to show they
generate the same subcategory. 

Suppose then that $X \in \cCad^{\wqce}$, and use the notation above for
its values. We must show that $X$ is trivial
in $\cCad$ if and only if it is cellularly trivial. It is obvious that
if $X$ is trivial then it is cellularly trivial. Suppose then that $X$ is
cellularly trivial in the sense that $[\aaa(g),X]_\ast = 0$ for all
$g\in \cG$. It follows that $[\aaa(A),X]_\ast=0$ for all $A\in \cC$.

Now consider the cofibre sequence (\ref{cof:fe}). Since $[f(N'),
e(V)]_*=0$ we see that 
$$[N',N']_*=[f(N'), f(N')]_*=[f(N'), X]_*=0$$
so $N'\simeq 0$ and it follows that $X\simeq e(V)$. But now
$$[V, V]_*=[e(V), e(V)]_*=[e(V), X]_*=0$$
so that $V\simeq 0$. Hence $X\simeq 0$ as required.

\end{proof}

\begin{remark}\label{rem:cGad}
\Cref{equiv2} shows that $\cGad$ generates the localizing subcategory
$\wqce$. Since $\cGad$ is a set of small objects,  it shows that
cellularization with respect to $\wqce$ exists and $R_{\wqce}\cCad
=R_{\cGad}\cCad$.  To avoid referring to any particular generators, we
henceforth write $R_{\wqce} \cCad$ for the adelic model of $\cC$. 
\end{remark}


\section{The separated and complete models}\label{sec:adelics}

In this section we construct  both the separated and complete models from
$\cCad$ by taking  a left localization of a cellularization of
$\cCad$.  We will introduce the relevant localizations and then  prove the following theorem.

\begin{thm} {\em (Adelic, separated and complete models)}
\label{thm:adsepcomp}
Let $\cC$ be a 1-dimensional Noetherian model category. Then there are Quillen equivalences 
$$ \cC \simeq_Q R_{\wqce} \cC_{ad} \simeq_Q L_{\Pi} R_{\wqcie} \cC_{ad} \simeq_Q L_{\Lambda}R_{\wqc} \cC_{ad}.$$
\end{thm}

The outline for this section is as follows. In Subsection \ref{subsec:cof} we will introduce the relevant
classes for cellularization, in Subsection \ref{subsec:PiLambda} we
introduce the left localizations. Subsection \ref{subsec:LR}
details the strategy,  and then in Subsection \ref{subsec:LRdiagram} we
assemble the pieces to prove the theorem.

\subsection{Classes of cofibrant objects}
\label{subsec:cof} 
We  introduce three classes of objects which will be classes of
cofibrant objects in our models. To show they can be the class of
cofibrant objects  we
need to establish they are generated by a set of small objects. 
All are weakly quasi-coherent ($\wqc$, a condition on the horizontal base change) but then we
require additional  conditions on the vertical base change (either
extendedness or the somewhat weaker condition of idempotent
extendedness). This gives the hierarchy
$$\wqc \supseteq \wqcie \supseteq \wqce.$$

\subsubsection{Weakly quasi-coherent objects}
\label{subsec:wqc}

In Section \ref{subsec:cellstrictholim}   we proved that the class of $\wqce$ objects
was generated by the set $\cGad$ (see Remark~\ref{rem:cGad}).
We now prove $\wqc$ is also generated by a set of small objects. 
 Since we are working stably, it is enough to show that the
acyclic objects in the cellularization are generated by a set of small
objects. 

We suppose $\xymatrix{N \ar@{..>}[r] &Q & \ar@{..>}[l] V}$ is $\wqc$ (i.e., $L_\gen N \simeq Q$) and then form 
$$
\xymatrix{
&V\ar@{..>}[dd]\ar[dr]|= &\\
                      &                 &V \ar@{..>}[dd] \\
N'\ar@{..>}[r]\ar[dr]&j_*V \ar[dr]&               \\
                     &N\ar@{..>}[r]       &Q
}$$
where $N'$ is constructed as the pullback. As we have assumed that
$\cC$ (and hence $(L_\gen \unit) \modules_\cC$) is right proper, it follows that $L_\gen N' \simeq j_\ast V$. In particular, the back face is a $\wqce$ object. Therefore the mapping cone of this morphism of diagrams is of the form $\xymatrix{M \ar@{..>}[r] &P & \ar@{..>}[l] 0}$
where $P \simeq L_\gen M$. It remains to observe 
any such object can be $\tensor$-built from the small $\wqc$ object $\xymatrix{\prod_\fm \Lambda_\fm \unit \ar@{..>}[r] &L_\gen \prod_\fm \Lambda_\fm \unit & \ar@{..>}[l] 0}$.

\subsubsection{Modules over product rings}
\label{subsubsec:modprod}
We are considering modules $M$ over product rings $\prod_i R_i$. Using
the projection $p_i: \prod_i R_i \lra R_i$ we may define an
$R_i$-module by extension of scalars
$$e_i M :=R_i \tensor_{\prod_iR_i}M . $$
This is idempotent in the sense that if we view $e_iM$ as an $\prod_i
R_i$-module by restriction $e_i e_i M\cong e_iM$ since $R_i
\tensor_{\prod_i R_i} R_i \cong R_i$.


Indeed, there is an adjunction 
$$\xymatrix{
 \left( \prod_i R_i \right) \modules \ar@{->}[r]<1.0ex> ^{\sigma}&
\ar@{->}[l]<1.0ex>^{\pi} \prod_i  \left( R_i \modules \right)
}$$
The right adjoint $\pi$ simply takes products of modules, and 
the left adjoint $\sigma$  is defined by $\sigma (M)_i=e_iM$.
For a module $M$ over $\prod_i R_i$ we write $M^{\Pi}=\pi
\sigma M$, and think of the unit of the adjunction $M\lra M^\Pi$ as a
form of  completion.


\subsubsection{Weakly quasi-coherent idempotent extended objects}
\label{subsec:wqcie}
We now consider the special case of product rings where the factors  are
the completions $\Lambda_\fm \unit$.

\begin{defn}
A map $V\lra Q$ of modules over
$L_\gen \unit \lra L_\gen \prod_\fm  \Lambda_\fm \unit$ is {\em idempotent
extended} if the idempotent pieces  $e_\fm j_\ast V
\stackrel{\cong}\lra  e_\fm Q$ of
the base change map are isomorphisms  for all closed points~$\fm$.  It
is {\em weakly idempotent extended} if the maps are weak
equivalences, and we write $\wqcie$ for the collection of all weakly
quasi coherent, weakly idempotent extended objects. 
\end{defn}

Since passage to idempotent pieces preserves weak equivalences, a
weakly extended object is weakly idempotent extended and $\wqcie
\supseteq \wqce$.

By similar arguments to Section \ref{subsec:cellstrictholim}  
and \ref{subsec:wqc} we
see that $\wqcie$ is 
the class of cofibrant objects in suitable cellularizations with 
respect to a set of small objects. 

We suppose $\xymatrix{N \ar@{..>}[r] &Q & \ar@{..>}[l] V}$ is $\wqcie$ and then once again form 
$$
\xymatrix{
&V\ar@{..>}[dd]\ar[dr]|= &\\
                      &                 &V \ar@{..>}[dd] \\
N'\ar@{..>}[r]\ar[dr]&j_*V \ar[dr]&               \\
                     &N\ar@{..>}[r]       &Q
}$$
As before, the mapping cone is of the form $\xymatrix{M \ar@{..>}[r]
  &P & \ar@{..>}[l] 0}$
where $P\simeq L_\gen M$, and we also see that $M$ is $\Pi$-trivial in the sense that $M$ is not necessarily equivalent to 0, but $M^\Pi\simeq 0$. 

It then remains to observe that any $\Pi$-trivial object is $\tensor$-built from the
$\Pi$-trivial object
$T:=\prod_\fm \Lambda_\fm \unit /\bigoplus_\fm \Lambda_\fm \unit $.  This
may not be small for all modules but $[T, \bigoplus_i N_i]=\bigoplus_i
[T, N_i]$ if $N_i$ is $\Pi$-trivial since $[T,N]=[\prod_\fm  \Lambda_\fm \unit ,N]$ when $N$ is
$\Pi$-trivial.

\subsection{Bousfield localizations}
\label{subsec:PiLambda}
We now explore the appropriate left Bousfield
localizations corresponding to the  right Bousfield localizations $R_\wqce, R_\qcie$ and
$R_\wqc$ as in Subsection \ref{intro:homotopical}.

Given an object $E=(\xymatrix{N \ar@{..>}[r] &Q & \ar@{..>}[l] V})$ of $\cCad$, we may consider the
set $E$-we of those maps $f$ so that $E\tensor f$ is a weak equivalence. We
note that this depends only on the objects $V, Q$ and $N$ appearing in the diagram and not on the base-change
maps between them. Since we work stably, it is enough to consider the set $\langle E\rangle$ of
$E$-acyclics.

\subsubsection{The $\Pi$ Bousfield localization}
We first consider a localization designed to ensure the bottom left 
object is a product of modules for the individual factors 
$\Lambda_{\fm}\unit$, to be used for the separated model. This will
be Bousfield localization with respect to 
$$
\Pi = \{ \xymatrix{\Lambda_{\fm}\unit \ar@{..>}[r] & L_\gen\Lambda_{\fm}\unit & L_\gen\unit \ar@{..>}[l]} \st \fm \mbox{ a closed point} \}
$$
We remark that the $\wqce$-cofibrant replacement is
$$\Gamma_{\wqce} (\Pi)  
\simeq \{ \xymatrix{\Lambda_{\fm}\unit \times L_\gen \Lambda_{\fm'} \unit \ar@{..>}[r] & L_\gen \prod_\fm \Lambda_\fm \unit & L_\gen\unit \ar@{..>}[l]}  \st \fm \mbox{ a closed point} \},  $$
where $\Lambda_{\fm'}\unit =\prod_{\fn \neq \fm}\Lambda_{\fn}\unit$. 

\subsubsection{The $\Lambda$ Bousfield localization}
\label{subsubsec:Lambda}
Now we consider a localization designed to ensure the bottom left 
object is a product of complete modules for the individual factors 
$\Lambda_{\fm}\unit$, to be used in the complete model. This 
will be Bousfield localization with respect to 
$$\Lambda =\{ \xymatrix{K_\fm \ar@{..>}[r] & 0 & L_\gen\unit \ar@{..>}[l]} \st \fm \mbox{ a closed point} \}$$
where $K_\fm$ is a Koszul object with support $\overline{\{ \fm\}}$  as used to define completion.

We remark that  the $\wqce$-cofibrant replacement is
$$\Gamma_{\wqce} (\Lambda) 
\simeq \{ \xymatrix{K_\fm \times L_\gen \prod_\fm \Lambda_{\fm} \unit \ar@{..>}[r] & L_\gen \prod_\fm \Lambda_\fm \unit & L_\gen\unit \ar@{..>}[l]}  \st \fm \mbox{ a closed point} \}. $$

It is clear that the set $\Pi$ generates a localizing subcategory
containing $\Lambda$. As such it follows that $L_\Pi =L_{\Pi\coprod \Lambda }$, and there
is a natural transformation $L_\Pi \lra L_\Lambda$.


\subsection{Quillen equivalences of left and right localizations}
\label{subsec:LR}

In this  subsection we describe the main strategy of proof for Theorem
\ref{thm:adsepcomp}: we provide sufficient conditions for various
composites of left and right Bousfield localizations to be Quillen equivalences. In the next section we will apply these criteria to establish the separated and complete
models. 

In fact we will be interested in zig-zags of left Quillen functors $$\cC\lla R\cC\lra L R\cC.$$
This zig-zag is obtained by considering the square
$$\xymatrix{
\cC\ar[d] & R \cC\ar[l] \ar[d] \\  
L \cC& L R \cC\ar[l]   \rlap{ .}  
}$$
We will give criteria for these localizations to be Quillen
equivalences. 

Recall that a left localization $\cC\lra L\cC$ is a Quillen equivalence if all the
maps that are inverted are already weak equivalences. We observe that it is
enough to check this on weak equivalences between cofibrant
objects, and the cofibrant objects are the same in both categories. Once again, for stable localizations it is enough to check this for acyclics. That is, we must show that if $x$ is a 
cofibrant object with $x\simeq 0$ in $L\cC$ then in fact $x\simeq 0$ in 
$\cC$.

Similarly to show that a right localization $R \D \lra \D$ is a Quillen equivalence we
must show that if $x\lra y$ is a cellular equivalence then it is an
actual equivalence. It is also sufficient to do this for fibrant
objects, and fibrant objects are the same in both categories. In the stable setting it is enough to do it for acyclics. That is, we must show that if $x$ is a  fibrant object with $x\simeq 0$ in $R\D$ then in fact $x\simeq 0$ in 
$\D$.


Of particular interest to this section will be those left Quillen functors of the form $LR\D \lra L\D$. Suppose we have a morphism  $f \colon x\lra y$
where $x$ is cofibrant in $LR\D$ and $y$ is fibrant in $L\D$. We must then show that $f$ is a weak equivalence in $LR\D$ if and only if it is a weak equivalence in $L\D$. We now assume that $L$ is a homological localization with respect to some object $E$. Observe that it is enough to check for $x$ cofibrant in $R\D$ and $y$ fibrant in $L\D$ that if $f$ is an equivalence in $LR\D$ then it is an equivalence in $L\D$. In particular, it suffices to show that if $E\tensor x \lra E\tensor y$ is an equivalence in $R\D$ then it is an equivalence in $\D$. 

\subsection{Cases at hand}
\label{subsec:LRdiagram}
For the cofibrant objects we have three candidates $\wqce \subseteq
\wqcie \subseteq \wqc$. We showed in Sections
\ref{subsec:cellstrictholim}, 
\ref{subsec:wqc} and \ref{subsec:wqcie}  that
each of these three classes is the sets of cofibrant objects 
in the cellularization with respect to a set of small objects. Accordingly, we can
form right Bousfield localizations with respect to each of them. 

For controlling the fibrant objects we wish to use the left Bousfield localizations introduced in \Cref{subsec:PiLambda}, namely $L_\Pi$ and $L_\Lambda$.

Assembling the pieces, one may construct the following diagram of left Quillen functors, where the marked equivalences are the things we shall now prove: 
$$\xymatrix@R=12mm{
\cCad \ar[d] & R_\wqc \cCad \ar[l] \ar[d] &  R_\wqcie \cCad \ar[l] \ar[d] &R_\wqce \cCad \ar[l] \ar[d]|-{\simeq}\\  
L_\Pi \cCad \ar[d] &L_\Pi R_\wqc \cCad \ar[d]\ar[l] & L_\Pi R_\wqcie \cCad \ar[l] \ar[d]|-{\simeq} &  L_\Pi 
R_\wqce \cCad \ar[l]|-{\simeq} \ar[d]|-{\simeq} \\  
L_\Lambda \cCad & L_\Lambda R_\wqc \cCad \ar[l]& L_\Lambda  R_\wqcie \cCad \ar[l]|-{\simeq}  &  L_\Lambda R_\wqce \ar[l]|-{\simeq} \cCad \\  
}$$

In the following discussion we will as usual have an object 
\begin{align*}
x= \left(
\begin{gathered} \xymatrix{
&V\ar@{..>}[d]\\
N \ar@{..>}[r]&Q 
} \end{gathered}
 \right)
\end{align*}
with  nub $N$, and vertex $V$. 

To establish the right hand vertical equivalences, since $\Lambda
\subseteq \Pi$, it suffices to prove that if a $\wqce$ object is $\Lambda$-trivial then it is
trivial.  Suppose then that $x$  is $\wqce$ and moreover
$\Lambda$-trivial. Since the vertex $V$ is at an initial point of the
diagram, we conclude that $V\simeq 0$. Since the object $x$ is weakly extended we have $j_\ast V \simeq Q\simeq 0$. Since it is weakly quasi-coherent, it follows that $N$ is torsion. The result then follows from the fact that a torsion object $N$  with $N \tensor K_\fm\simeq 0$ for all
$\fm$ is contractible. 

To establish the right hand horizontal equivalence in the middle row 
we need to show that if 
we have a morphism $x \to y$ where $x$ is $\wqce$, $y$ is $\wqcie$ and $x\tensor \pi \lra y\tensor \pi$ is a 
$\wqce$-cellular equivalence for all $\pi \in \Pi$ then it is a $\wqcie$-cellular equivalence.

It is convenient to factor $x\lra y$ as $x\lra y^c \lra y$ where $y^c$ is 
$\wqce$ (i.e., the cofibrant replacement of $y$ in $R_{\wqce} \cCad$). Now the map $y^c\lra y$ is an equivalence at the vertex,
an i-equivalence at the $L\Lambda \unit$-module, and a
$\Pi$-equivalence at the nub, so it is a $\Pi$-equivalence. In other words we may assume both 
$x$ and $y$ are $\wqce$, and since we are in the stable setting, 
we need only prove for $x$ $\wqce$ that $[k,x\tensor \pi]=0$ for $k$ $\wqce$
implies it also holds for $k$ $\wqcie$.  If $k$ is $\wqcie$ then one checks
that the $\wqce$-cellularization $k^c \lra k$ induces an isomorphism in $[- , x 
\tensor \pi]$ for $\pi \in \Pi$. This is true because 
$[k, x\tensor \pi]=[k\tensor \pi , x\tensor \pi]$, and $k^c\lra k$ is
a $\pi$-equivalence. 

To establish the lower horizontal equivalences we need to show that if 
$x$ is $\wqce$, and $y$ is $\wqc$ and $x\tensor \pi \lra y\tensor \pi$ is a 
$\wqce$-cellular equivalence then it is a $\wqc$-cellular equivalence. It 
is once again convenient to factor $x\lra y$ as $x\lra y^c\lra y$ where $y^c$ is 
$\wqce$. Now the map $y^c\lra y$ is an equivalence at the vertex, and a 
$\Lambda$-equivalence at the $\Lambda \unit$
spot, so it is a $\Lambda$-equivalence. In other words we may assume both 
$x$ and $y$ are $\wqce$, and since we are in the stable setting, 
we need only prove for $x$ $\wqce$ that $[k,x\tensor \lambda]=0$ for $k$ $\wqce$ for all $\lambda \in \Lambda$ 
implies it also holds for $k$ $\wqc$.  If $k$ is $\wqc$ then the 
$\wqc$-cellularization $k^c \lra k$ induces an isomorphism in $[- , x 
\tensor \lambda]$. 

Indeed, suppose 
\begin{align*}
k&=(\xymatrix{M \ar@{..>}[r] &P & \ar@{..>}[l] U}),\\ x&=(\xymatrix{N \ar@{..>}[r] &Q & \ar@{..>}[l] V}),\\
\lambda &=(\xymatrix{K_\fm \ar@{..>}[r] &0 & \ar@{..>}[l] L_\gen \unit }).
\end{align*}
Then  $[k, x\tensor \lambda]=[M,
N\tensor K_\fm]\times [U, V]$. Since $k^c\lra k$ is an isomorphism at
the $L_\gen \unit$ vertex, there is nothing to check at that point. At the $\prod_\fm \Lambda_\fm \unit$ position, we
see the fibre of $M^c \lra M$ is the same as the fibre at $P^c\lra P$,
and hence in particular it is $L_\gen X$ for some $X$, and $[L_\gen X, N\tensor K_\fm]=0$.

This completes the proof of Theorem \ref{thm:adsepcomp}.

\subsection{Bifibrant objects}
In summary, we now have three homotopical models:

\begin{itemize}
\item The adelic model -- $R_\wqce \cCad$, with bifibrant objects $\wqce$
\item The separated model -- $L_\Pi R_\wqcie \cCad$, with bifibrant
  objects $\mathrm{wpqcie}$ (i.e., $\wqcie$ and the nub equivalent to a product)
\item The complete model -- $L_\Lambda R_\wqc \cCad$, with bifibrant
  objects $\wqck$ (i.e., $\wqc$, with the nub complete).
\end{itemize}

To finish off \Cref{part1} we will describe explicit constructions of moving between bifibrant replacements of one of these models to another.

To obtain a separated or complete bifibrant object from an adelic bifibrant object
(i.e., a $\wqce$-object $\xymatrix{N \ar@{..>}[r] &Q & \ar@{..>}[l] V}$) we form the diagram
$$
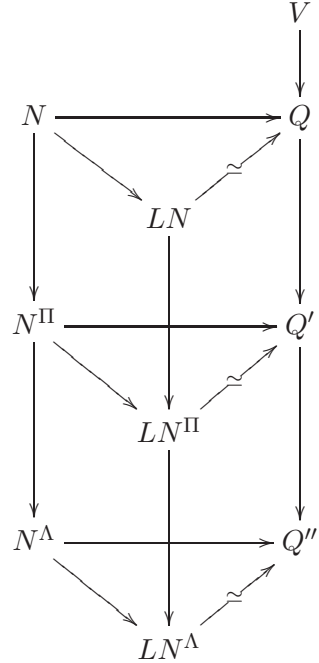
\xymatrix{&&V\ar[d]\\
N\ar[dd]\ar[rr]\ar[dr]&&Q\ar[dd]\\
&LN\ar[dd]\ar[ur]|{\simeq}&\\
N^{\Pi}\ar[dr]\ar[rr]\ar[dd]&&Q'\ar[dd]\\
&LN^{\Pi}\ar[ur]|{\simeq}\ar[dd]&\\
N^\Lambda \ar[dr]\ar[rr]&&Q''\\
&LN^\Lambda \ar[ur]|{\simeq}&
}$$
where the front right faces define $Q'$ and $Q''$ as homotopy pushouts
and superscripts $\Pi$ and $\Lambda$ refer to fibrant replacements. One
sees that the map $Q\lra Q'$ is an idempotent equivalence since that holds for
$N\lra N^\Pi$ and this is preserved by localization. 

Alternatively, given a bifibrant model in the complete or separated
model, we need to upgrade the vertical from a map or an 
idempotent extended map and to a fully extended map, and these operate by taking homotopy pullbacks on the
back face. The fact that these are homotopy pullbacks was discussed in
the overall introduction.

To go from the complete to the separated
model we first go to the standard model and then to the separated
model. 

\part{Algebraic models}\label{part2}

\section{Skeletons and the roadmap}
\label{sec:roadmap}

In \Cref{part2} we show how in certain
algebraic situations we can replace the underlying category $\cCad$ of
the model by a more economical skeleton.  A  skeleton of
a model category is a Quillen equivalent subcategory $\cS$ in which homotopy
equivalences are closer to isomorphisms, so that the  homotopy relation
has less work to do.

\subsection{Skeletons}
The simplest case is when  the inclusion $i \colon \cS \lra \cC$ has a {\em 
  right} adjoint, that is, $\cS$ is \emph{coreflective} in $\cC$.  The subcategory $\cS$ then inherits completeness and 
cocompleteness 
from $\cC$.
 If 
the generating sets of  cofibrations and acyclic cofibrations of $\cC$
lie in $\mathcal{S}$ there is a left-lifted model structure on
$\mathcal{S}$  with the same classes of weak equivalences,
cofibrations and fibrations~\cite{coreflective} and the inclusion is a
Quillen equivalence: we say $\mathcal{S}$ is a 
\emph{cellular skeleton} of $\cC$. 

If the inclusion $i: \cS \lra \cC$ has a {\em 
  left} adjoint, that is, $\cS$ is \emph{reflective} in $\cC$, then
there is a little more to do.  The subcategory $\cS$ again inherits
completeness and cocompleteness  from $\cC$. For a simple example of this, the
reader can look ahead to Section~\ref{subsec:nubsepskeleton}, where we
describe $\prod_\fm (R_\fm \modules)$ as a reflective subcategory of
$\prod_\fm (R_\fm \modules)$.

 Under suitable hypotheses, if 
 $\cC$ is cofibrantly generated
 we  can then lift the model structure in $\cC$ along
the right adjoint $i$ to obtain a model structure on $\cS$. 
If the inclusion of $\cS$ in  $\cC$ is a Quillen equivalence, 
we say $\mathcal{S}$ is a  \emph{local skeleton} of $\cC$.  

For example, if the left adjoint takes small objects to small objects, and if the unit and
counit are equivalences on generators, we obtain the conclusion from 
the Cellularization Principle. However,  in our principal examples the left
adjoint does not preserve small objects and we shall see in Section~\ref{subsec:qcmodels} that we must argue more directly.

\subsection{The roadmap}
\label{subsec:roadmap}
We may now describe our strategy. In Part 1 we constructed three
different models for $\cC$ with underlying category $\cCad$. We are going to
vary the underlying category and find models for $\cC$ on each one. In
fact we will construct a diagram of adjoint pairs in the pattern 
$$
\xymatrix{
\cCadqce 
\ar@<-.75ex>[d]_{i} 
& \cCsepqcie \ar@<-.75ex>[d]_{i} &\\
\cCadqc 
\ar[u]<-1.0ex>_{\Gamma_{\mathrm{e}}}\ar@<-.75ex>[d]_{i}
\ar@<.75ex>[r]^{\sigma_*} &
 \cCsepqc \ar[u]<-1.0ex>_{\Gamma_{\mathrm{ie}}}
\ar@<.75ex>[l]^{\sigma^*}
\ar@<-.75ex>[d]_{i} 
\ar@<.75ex>[r]^{\ell_*} 
&\cCLkqc  \ar@<-.75ex>[d]_{i} \ar@<.75ex>[l]^{\ell^*}\\
\cCad 
\ar[u]<-1.0ex>_{\Gamma_{\qc}} & \cCsep \ar[u]<-1.0ex>_{\Gamma_{\qc}}
&\cCLk \ar[u]<-1.0ex>_{\Gamma_{\qc}}
}$$
The categories $\cCad, \cCsep$ and $\cCLk$ at the bottom are categories of diagrams
on the cospans $\ad, \sep$ and $\Lk$ to be introduced below. 
Above each of the three there are subcategories.
In each case, the letter $i$ indicates a coreflective inclusion with
right adjoint  $\Gamma$ (with appropriate subscript). The functors
$\Gamma$  are introduced in Section \ref{sec:cellularskeletons}. 
The horizontal adjoint pairs are induced from adjoint pairs at one
vertex of the cospan by a construction to be described in Subsection
\ref{subsec:qcpush}. We will show that all the categories in the
diagram except $\cCsep$ and $\cCLk$ inherit model structures lifted
from  that on $\cCad$.

The diagram serves as a roadmap. 
The minimal skeletons of the adelic, separated and complete models are at the tops of the columns. 
The fact that vertical inclusions $i$ are Quillen equivalences follows
by the cellular skeleton argument above: this only requires us to show that the model
structures are generated by cofibrations and acyclic cofibrations from
the top category in each column. In effect we need to show an object
for which the $\qc$, $\qce$, $\qcie$ condition holds weakly is equivalent to a subobject for which is holds strongly.

For the horizontal adjunctions, the categories decrease in size as we
move right. To show that they are inclusions of local skeletons, we
use corresponding generators and show that the unit and counit are 
weak equivalences. 

\subsection{The equivalences} Here we shall give an overview of the argument: full
details are provided in Section~\ref{sec:skeletonequiv}.

If we start from the adelic model $R_\wqce\cCad$ 
at the bottom left and the passage to cellular skeletons in the first 
column shows 
$$R_\wqce\cCad=R_\we R_\wqc \cCad\simeq_Q R_\we \cCadqc\simeq_Q
\cCadqce,$$ 
and the cellular skeleton of the adelic model is $\cCadqce$.

If we start from the separated model $L_\Pi R_\wqcie\cCad$ we  have the equivalences  
$$L_\Pi  R_\wqcie\cCad =L_\Pi R_\wie R_\wqc\cCad  
\stackrel{(1)} \simeq_Q L_\Pi R_\wie \cCadqc  
\stackrel{(2)} \simeq_Q R_\wie \cCsepqc  
\stackrel{(3)} \simeq_Q \cCsepqcie$$ 
where (1) and (3) are cellular skeleton equivalences and (2) is a local  
skeleton equivalence. The skeleton of the separated model is $\cCsepqcie$.

If we start from the complete model $L_\Lambda R_\wqc\cCad$ we  have the equivalences  
$$L_\Lambda   R_\wqc\cCad \stackrel{(1)} \simeq_Q L_\Lambda  \cCadqc = L_\Lambda L_\Pi \cCadqc  
\stackrel{(2)} \simeq_Q L_\Lambda \cCsepqc  
\stackrel{(3)} \simeq_Q \cCLkqc$$ 
where (1) is a cellular skeleton equivalence, whilst (2) and (3) are local  
skeleton equivalences. The skeleton of the complete model is $\cCLkqc$.

It remains to introduce the categories and functors appearing in the
roadmap diagram, which we do in Sections \ref{sec:adsepLk} and~\ref{sec:cellularskeletons}. After this we shall show
that objects with weak conditions are equivalent to ones with strong
conditions. Finally we establish that the unit and counits of the horizontal adjunctions are
derived equivalences on generators: this  will be done in Section \ref{sec:skeletonequiv}. 

\section{Local skeletons of the nub}
\label{sec:nubskeletons}
As described above, there are two ways of moving to a smaller
skeleton. One of them changes the categories in the diagram and one of them
places restrictions on the objects for a fixed diagram. The change of
categories in the diagram is passage to a local skeleton, and is based  on changing the  nub category $(\prod_\fm
\Lambda_\fm \unit)\modules$ to a local skeleton. In this section we
introduce the functors we use to change the nub and in Section
\ref{sec:adsepLk} we explain how to lift this to diagrams. 

\subsection{Pattern}
In both cases the argument proceeds as follows. We have an inclusion
$i: \cN' \lra \cN$ of a subcategory with a left adjoint $F$, and a
left Bousfield localization $L$ of $\cN$. In our situation this gives a diagram of
left Quillen functors
$$
\xymatrix{
\cN \ar[d] \ar[r]^F &\cN'\\
L\cN \ar[ur]_{\simeq_Q}^{\overline{F}} & 
}$$
In other words, we have a Quillen pair $F\dashv i$ relating $\cN$ to
$\cN'$, but $F$ is also a left Quillen functor from $L\cN$ and the
same adjoint pair gives a Quillen equivalence $L\cN
\simeq_Q \cN'$.

\subsection{The separated skeleton}
\label{subsec:nubsepskeleton}
Let $R_\fm = \Lambda_\fm \unit$ in a Noetherian model category $\cC$. There are many modules over the product ring $\prod_\fm \Rm$, but the
ones we care most about are those which are products of modules over
the individual rings $\Rm$. Using the projective model
structure, 
$\prod_\fm (\Rm\modules)$ is a local skeleton of $(\prod_\fm \Rm)\modules$.

As in Subsection \ref{subsubsec:modprod}, the functor 
$$\sigma: (\prod_\fm \Rm ) \modules \lra \prod_\fm 
(\Rm \modules). $$ 
defined by 
$$\sigma(M)_\fm =e_\fm M := \Rm \tensor_{\prod_\fm
  \Rm} M$$
 has right adjoint 
$$\pi : (\Rm \modules)\lra (\prod_\fm
\Rm) \modules$$
defined by 
$$\pi (\{N_\fm\})=\prod_\fm N_\fm. $$
To see that $\sigma$ is the inclusion of a reflective subcategory we
note
$$N_\fm \stackrel{\cong}\lra R_\fm \tensor_{\prod_\fm \Rm} \prod_\fm N_\fm$$
is an isomorphism.

\begin{lemma}
\label{lem:sigmapi}
The functor $ \sigma: (\prod_\fm \Rm ) \modules \lra \prod_\fm 
(\Rm \modules)$ is the inclusion of a reflective subcategory and
$\prod_\fm (\Rm
\modules)$ is a local skeleton of  $L_\Pi (\prod_\fm \Rm ) \modules.$
\end{lemma}

\begin{proof}
We take the projective  model structure on $(\prod_\fm
\Rm)\modules$ and similarly on each of the factors of  $\prod_\fm (\Rm
\modules)$. We see that $\sigma \dashv \pi$ is a Quillen adjunction,
and we may lift the model structure along the right adjoint.

Indeed, this same adjunction gives a Quillen pair $\sigmab: L_\Pi
(\prod_\fm \Rm)\modules \lra \prod_\fm( \Rm\modules)$. One checks directly
that this is a Quillen equivalence. 
\end{proof}

\subsection{The L-complete skeleton}
\label{subsec:nubLkskeleton}
Unlike the separated skeleton, the L-complete skeleton will only exist under some hypothesis on $\Lambda_\fm \unit$. As such, we now restrict to the situation that $\Rm$ is a graded commutative Noetherian local  ring with maximal
ideal $\fm$. We would like to 
replace $\Rm\modules$ by a skeleton which is complete in some sense. 
The first thought is  to replace $\Rm \modules$
 by its Bousfield completion $\Lambda_\fm (\Rm\modules)$
however this is not helpful for the  skeleton since the underlying subcategory remains the
same. 

It is natural to seek an abelian model. One thinks first of  the category of $\fm$-adically 
complete modules, but since $\fm$-adic completion is not exact, the 
category is not well behaved.  Nonetheless, Pol and Williamson have shown the completion
 is the derived category of a convenient abelian category.  
Writing  $\Lfm$ for the 0th left derived functor $\Lfm$
of $\fm$-adic completion, we take   the $\Lfm$-complete modules
$M$ (those for which the natural transformation $M\lra \Lfm M$ is an
isomorphism \cite{GreenleesMay92}). The category $\RmmodLfmhat$  of differential
graded $\Lfm$-complete modules is a very well behaved reflective
subcategory of $\Rm\modules$ and gives a local skeleton. The following
result is stated in \cite[Appendix A]{HoveyStrickland99},
\cite[Appendix A]{BarthelFrankland}; both references
have a blanket assumption that their rings are regular and local, but 
these assumptions are unnecessary for the given proofs of the following statement.
 
\begin{lemma}
\leavevmode
\label{lem:Lfmreflector}
\begin{itemize}
    \item The category $\RmmodLfmhat$ is abelian, and the inclusion
      functor $i \colon \RmmodLfmhat \hookrightarrow \Rm \modules$ is exact.
    \item The  inclusion functor has left  adjoint $\Lfm$.
    \item The category $\RmmodLfmhat$ is complete and cocomplete.
\end{itemize}
\end{lemma}

 On the basis of this we have an abelian local skeleton of $\Lambda_\fm(\Rm \modules)$.

\begin{prop}[{\cite[6.10]{PolWilliamson}}]
\label{prop:Lkmodel}
There is a cofibrantly generated monoidal model structure on
$\RmmodLfmhat$ where a map is a weak equivalence or fibration if it is when
viewed in $\Rm \modules$: the model structure is right-lifted from that in
$\Rm \modules$.

Moreover, there is a symmetric monoidal Quillen equivalence
$$
L_0^\fm : \Lambda_\fm (\Rm \modules )\rightleftarrows \RmmodLfmhat  : i.
$$
\end{prop}

We note that  $\Lfm$ does not preserve sequential
colimits, so this requires a direct proof rather than using the
general right-lifting criterion. 

\section{The adelic, separated and L-complete categories}
\label{sec:adsepLk}

Having described two left adjoints  at the nub category, we need to explain how
to this may be extended to the whole cospan. In Section~\ref{subsec:qcmodels} we describe how to equip these categories with relevant model structures.

\subsection{Quasi-coherent diagram categories}
\label{subsec:qcdiagram}
We will name and display the left adjoints, with notation
for the right adjoints following from it, so that if  $f: \cC \lra
\cC'$ is a left adjoint the right adjoint will be called $f^r$.

We start from a cospan of left adjoints, but since we focus on
quasi-coherent objects,  the categories do not
play equal roles, and the notation reflects this. The cospan of left
adjoints is the diagram $\alpha =(\cN \stackrel{v}\lra \cQ
\stackrel{h}\lla \cV)$,   where $h$ and $v$ stand for `horizontal'
and `vertical' and we have the corresponding  category 
$$
\cCal  =\cCal (\cN, h)= \bbGamma \left( \begin{gathered}\xymatrix{
& \cV \ar[d]^{v} \\ 
 \cN  \ar[r]_h& \cQ 
}  \end{gathered} \right) 
$$
of diagrams. The vertical functor $v: \cV\lra \cQ$ will remain
constant throughout the discussion so is not recorded in the
notation.  Given a left adjoint $F: \cN \lra \cN'$ we may form a new
cospan
$$F_*\alpha =(\cN'\stackrel{F^r} \lra \cN \stackrel{h} \lra  \cQ \lla \cV)$$
where the horizontal $h$ has been replaced by the composite $hF^r$. We
note that the horizontal of $F_*\alpha$ is usually not a left
adjoint. Nonetheless, we can consider the category of diagrams
$$
\cCFal =\cCFal (\cN', hF^r) = \bbGamma \left( \begin{gathered}\xymatrix{
&& \cV \ar[d]^{v} \\ 
 \cN'\ar[r]_{F^r}& \cN  \ar[r]_h& \cQ 
}  \end{gathered} \right) 
$$
\subsection{Quasi-coherent pushout along a left adjoint}
\label{subsec:qcpush}
Our main focus is the category 
$$
\cCalqc =\cCalqc( \cN, h)= \bbGamma \left( \begin{gathered}\xymatrix{
& \cV \ar[d]^{v} \\ 
 \cN  \ar@{->}[r]_h|-{\text{\Large \textbullet}}& \cQ 
}  \end{gathered} \right) 
$$
of $\qc$ objects.  Since the base change along $h$ is an isomorphism, we
may  think of a $\qc$ object 
$$
Y =
 \left( \begin{gathered}\xymatrix{
& V \ar@{..>}[d]^{v} \\ 
N \ar@{..>}[r]|-{\text{\Large \textbullet}}& Q 
}  \end{gathered} \right) 
$$
as the nub $N$ together  with the additional structure of a vertex $V$
and a map $i: v(V)\lra h(N)$ in the splicing category $\cQ$ .

We describe how a left adjoint $F: \cN \lra \cN'$ induces comparisons between $\qc$ diagrams
 with different nubs.

\begin{lemma} 
\label{lem:qcpush}
A left adjoint $F:\cN \lra \cN'$ induces a
left adjoint $F_*: \cCal^{qc} \lra \cCFal^{qc}$. If $F$ is the inclusion of a
reflective subcategory, so is $F_*$.
\end{lemma}

\begin{proof}
The left adjoint $F_*$ is defined by 
$$F_*Y=
\left( \begin{gathered}\xymatrix{
& V \ar@{..>}[d]^{v} \\
&hN \ar[d]^{h\eta}\\
FN \ar@{..>}[r]|-{\text{\Large \textbullet}}& hF^rFN 
}  \end{gathered} \right) 
$$
where $\eta : N \lra F^r FN$ is the unit of the adjunction. The right
adjoint $F_*^r=F^*$ is defined by 
$$F^*Y' =
\left( \begin{gathered}\xymatrix{
& V'\ar@{..>}[d]^{v} \\
F^rN' \ar@{..>}[r]|-{\text{\Large \textbullet}}& hF^rN'
}  \end{gathered} \right) 
$$
The unit  $\hat{\eta} : Y \lra F^*F_*Y$  of the adjunction is given by the 
diagram 
$$\xymatrix{
&V\ar[r]^=\ar@{..>}[d]&V\ar@{..>}[d]\\
N\ar@{..>}[r]\ar[d]&hN\ar[r]^=\ar[dr]&hN\ar[d]\\
F^rFN\ar@{..>}[rr]&&hF^rFN
}$$
The counit  $\hat{\epsilon} : F_*F^*Y'\lra Y'$  of the adjunction is given by the 
diagram 
$$
\xymatrix{
&&V' \ar[dr]^=\ar@{..>}[dd]&\\
&&&V'\ar@{..>}[ddd]\\
&&hF^rN'\ar[ddr]\ar[d]&\\
FF^rN' \ar@{..>}[rr]\ar[dr]&&hF^rFF^rN'\ar[dr]&\\
&N'\ar@{..>}[rr]&&hF^rN'
}$$
where commutativity of the triangle is given by the triangular identity of the
original adjunction. The triangular identities for
$\etahat$ and $\epshat$  follow directly from the triangular
identities of $\eta$ and $\epsilon$ at the nub.
\end{proof}

\subsection{The separated cospan}
\label{subsec:sepcospan}
The  starting point is the adelic diagram $\alpha =\ad$, so that in
$\cCadqc$,  the nub is a module over the product ring $\prod_\fm \Rm$ where $\Rm =\Lambda_\fm \unit$. In the
expanded notation, 
$$\cCadqc =\cCadqc ((\prod_\fm \Rm ) \modules, L_\gen).$$ 

We now use the left adjoint
$$\sigma: (\prod_\fm \Rm ) \modules \lra \prod_\fm 
(\Rm \modules) $$
from Subsection \ref{subsec:nubsepskeleton}.

\begin{defn}
The category $\cCsep$ is the category of sections of the \emph{separated cospan}: 
$$
\cCsep = \cC_{\sigma_* ad}=\bbGamma \left( \begin{gathered}\xymatrix{
& (L_\gen \unit) \modules_{\cC} \ar[d]^{j_\ast} \\  \prod_\fm
\left[ \Lambda_\fm \unit \modules_\cC \right]  \ar[r]_-{L_\gen \pi} & \left(L_\gen  \prod_\fm \Lambda_\fm \unit \right) \modules_\cC
}  \end{gathered} \right)
$$
Unpacking the definition, the horizontal functor $L_\gen\pi$ takes a collection of modules
$\{M_\fm\}_\fm$ to the module 
$L_\gen \prod_\fm M_\fm$, so that an object of $\cCsep$ is a diagram of the form
$$ X=
\left[ \begin{gathered} \xymatrix{
& U \ar@{..>}[d] \\ \{M_\fm\}_{\fm} \ar@{..>}[r]& P 
} \end{gathered} \right]
$$
Its base change maps are $j_\ast U \to P$ and $L_\gen \prod_\fm M_\fm \to P$.
\end{defn}

Our model is built from the $\qc$ objects $\cCsepqc$ in $\cCsep$. The 
construction of Section \ref{subsec:qcpush} shows that 
the $\sigma \dashv \pi$ adjunction induces an adjunction with left adjoint 
$$\sigma_*: \cCadqc=\cCadqc (L_\gen)\lra \cCsepqc(L_\gen 
\pi)=\cCsepqc. $$

\subsection{The L-complete cospan}
\label{subsec:Lkcospan}
We  use the $L$-complete local skeleton of Subsection \ref{subsec:nubLkskeleton}
and the constructions of Section \ref{subsec:qcpush} to 
replace the nub $\prod_\fm (\Lambda_\fm \unit \modules)$ by a complete 
category, but we need to restrict the categories to the appropriate
context.

We restrict attention to the special case where individual categories 
$\Lambda_\fm (\cCm)$ are algebraic in a strong sense.  

\begin{defn}
We say  that $\cC$ is {\em formally algebraic} if, for each closed 
point $\fm$, $\Lambda_\fm \unit $ is formal in the sense that we can find a Noetherian graded commutative local ring $(R_\fm, \fm)$ and a Quillen equivalence $\Lambda_\fm \unit \modules_\cC \simeq \Rm \modules$.

 Note that we have abused notation by using $\fm$
to refer to both the Balmer prime $\fm \in \spc^{\omega}(\cC)$ on the 
left, and the algebraic prime $\fm \in \spec (\Rm)$ on the 
right. 
\end{defn}

\begin{example}\leavevmode
\begin{itemize}
\item[(i)] The most  obvious example is if $\cC=R\modules$ of (chain complexes of) $R$-modules for a 1-dimensional 
commutative Noetherian domain, so that all the completions come from a 
single ring: in this case $(\Rm,\fm)=(R_\fm^\wedge, \fm)$.  The results are valuable 
even in this case.
\item[(ii)] We will also consider the example when $\cC$ consists of
  (chain complexes of)
quasi-coherent sheaves over a curve $C$. In this case $\fm$ runs
through closed points and $\Lambda_\fm \cO_C\simeq
(\cO_C)_\fm^{\wedge}$ is a skyscraper sheaf, and we may take  $\Rm
=\Gamma (C; (\cO_C)_\fm^{\wedge})$. 
\end{itemize}
\end{example}

As in Subsection \ref{subsec:nubLkskeleton}, we let $\RmmodLfmhat$
denote the category of differential
graded $\Lfm$-complete modules. We  consider the product
$$\ell: \prod_\fm(\Rm \modules) \lra \prod_\fm 
(L_0^\fm-\Rm \modules)$$
 of the individual left adjoints 
$$\Lfm : \Rm \modules \lra  L_0^\fm-\Lambda_\fm \unit\modules. $$

We may now define a  skeleton of the complete model. 
\begin{defn}
Suppose $\cC$  is formally algebraic. We
define the \emph{$L$-complete cospan} $\cC_{\ell_* ad}$ and the associated
category of diagrams
$$
\cCLk = \bbGamma 
\left( \begin{gathered}
\xymatrix{
& (L_\gen \unit) \modules \ar[d]^{j_\ast} \\ 
 \prod_\fm \left[\RmmodLfmhat \right]  \ar[r]_-{L_\gen \ell} &
 (L_\gen  \prod_\fm \Lambda_\fm \Rm)\modules
} 
 \end{gathered} \right)
$$
\end{defn}

Our model is built from the qc objects $\cCLkqc$ in $\cCLk$. The 
construction of Section \ref{subsec:qcpush} shows that the left adjoint
$\ell$ induces a left adjoint 
$$\ell_*: \cCsepqc=\cCsepqc (L_\gen\pi )\lra \cCLkqc(L_\gen \pi \ell) =\cCLkqc. $$

\section{The cellular skeletons}\label{sec:cellularskeletons}
We have explained that given a coreflective inclusion $i: \cS \lra
\cC$ of a subcategory, then if $\cC$ has a cofibrantly generated model
structure with generators in $\cS$ then by \cite{coreflective} the
inclusion is an equivalence. To fill in the vertical equivalences in
the roadmap of Section \ref{sec:roadmap} it remains only to construct
the right adjoints and observe the generators come from the smaller
subcategories.

\subsection{Quasi-coherification}
\label{subsec:Gqc}
 The quasi-coherification
$\Gamma_{qc}X$ of an object
\begin{align*}
X= \left[ 
\begin{gathered} \xymatrix{
& V \ar@{..>}[d] \\ N \ar@{..>}[r] & Q 
} \end{gathered}
 \right]
\end{align*}
of $\cCad$ is defined by
\begin{align*}
\Gamma_{\qc} \left[ 
\begin{gathered} \xymatrix{
& V \ar@{..>}[d] \\ N \ar@{..>}[r] & Q 
} \end{gathered}
 \right]
 &= 
 \left[
 \begin{gathered}
 \xymatrix{
&\textcolor{red}{V'}  \ar[r] \ar@{}[dr]|<<{\lrcorner}  \ar@{..>}@[red][d] & V \ar@{..>}[d] \\ \textcolor{red}{N} \ar@{..>}@[red][r]|-{\textcolor{red}{\text{\Large \textbullet}}} & \textcolor{red}{L_\gen N} \ar@{->}[r] & Q
}
\end{gathered}
 \right]
  = 
  \left[ 
\begin{gathered} \xymatrix{
& V' \ar@{..>}[d] \\ N \ar@{..>}[r]|-{\textcolor{black}{\text{\Large \textbullet}}}  & L_\gen N 
} \end{gathered}
 \right]
\end{align*}
As before,  a solid bullet denotes an isomorphism after base change. In more detail, the dotted arrow from $N$ to $Q$ means we have a map
$L_{\gen}N \lra Q$ and we obtain $V'$ by pulling back along it. The
object $\Gamma_{\qc}$ is highlighted in red. 

The two-coloured display in the definition shows the counit $i\Gqc
X\lra X$. It is also obvious that if $X$ is qc then $\Gqc i X =X$. The
triangular identities are immediate, so that we have an adjunction
$$\xymatrix{
\cCadqc \ar@<1.0ex>^{i}[r]&\cCad .\ar@<1.0ex>^{\Gqc}[l]
}$$

The definition of $\Gqc$ on $\cCsep$ and $\cCLk$ is exactly similar,
and it is easy to see they are also right adjoint to the
inclusion. This gives the bottom half of the road map in Subsection \ref{subsec:roadmap}.

\subsection{Extendification}
The definition of $\Ge$ is very much like that of $\Gqc$:
\begin{align*} 
\Gamma_{\mathrm{e}} \left[
\begin{gathered} \xymatrix{
& V \ar@{..>}[d] \\ N \ar@{..>}[r] & Q
} \end{gathered}
 \right]
 =
 \left[
\begin{gathered} \xymatrix{
& \textcolor{red}{V} \ar@{..>}@[red][d]|-{\textcolor{red}{\text{\Large \textbullet}}}  \\ \textcolor{red}{N'} \ar@{..>}@[red][r] \ar[d] & \textcolor{red}{j_\ast V} \ar@{->}[d] \\ N \ar@{..>}[r] & Q
} \end{gathered}
 \right]
   = 
  \left[ 
\begin{gathered} \xymatrix{
& V \ar@{..>}[d]|-{\textcolor{black}{\text{\Large \textbullet}}} \\ N' \ar@{..>}[r]  & j_\ast V 
} \end{gathered}
 \right]
 \end{align*}
Once again, $N'$ is defined as a pullback and $\Gamma_{\mathrm{e}}X$ is
highlighted in red.  Once again, the two coloured display in the definition shows the counit $i\Ge
X\lra X$. It is also obvious that if $X$ is extended then $\Ge i X =X$. The
triangular identities are immediate, so that we have an adjunction
$$\xymatrix{
\cCade\ar@<1.0ex>^{i}[r]&\cCad .\ar@<1.0ex>^{\Gqc}[l]
}$$

\subsection{The quasi-coherification extendification  functor}
In order to construct an adjoint to the inclusion $\cCadqce \lra
\cCadqc$ we need an additional condition. 

The Bousfield localization $L_\gen$  is smashing in the sense that 
$L_\gen X\simeq L_\gen  \unit \tensor X$, and  it follows that multiplication gives a weak 
equivalence $L_\gen \unit \tensor L_\gen \unit \simeq L_\gen \unit$.  When forming 
algebraic models, we may obtain a leaner model by requiring this to hold more 
strictly. 

\begin{defn}\label{defn:strictly_smashing}
We say that a localization functor $L$ is {\em strictly smashing} if 
the multiplication map $L\unit \tensor L\unit \lra L\unit$ is an 
isomorphism and $L\unit$ is flat. 
\end{defn}

\begin{example}\leavevmode
\label{eg:multclosed}
\begin{itemize}
\item [(i)] If $\cC$ is the category of (chain complexes of) $R$-modules for a commutative ring $R$ and 
$L$ is localization to invert a multiplicatively closed set, then $L$
is strictly smashing. 
\item[(ii)] If $\cC$ is the category of (chain complexes of) quasi-coherent sheaves on a curve $C$
the generic localization
$L_\gen$ is passage to stalks over the generic point, and this is
also strictly smashing. 
\end{itemize}
\end{example}

\begin{defn}
Let $\cC$ be a 1-dimensional Noetherian model category. We will say that $\cC$ has {\em strict generization} if  $L_\gen$ is strictly smashing. 
\end{defn}

All the examples we have in mind  are based on categories of 
differential graded modules, so they are simplicial and have strict 
generization essentially as in Example \ref{eg:multclosed}.

\begin{lemma}
\label{lem:Gammaepreservesqc}
If we have strict generization, then the functor  $\Ge$ preserves $\qc$-modules. Hence 
$\Gqce=\Ge \Gqc$ is the right adjoint to the inclusion $i:\cCadqce
\lra \cCad$.  
\end{lemma}


\begin{proof}
It is clear that $\Gamma_{\qc}$ lands in the category where the horizontal base change map is an isomorphism. As such, it is enough to assume that we have $X \in \cCadqc$ and that applying $\Gamma_{\mathrm{e}}$ preserves the property of being $\qc$. We set
$$X =  \left[ 
\begin{gathered} \xymatrix{
& V \ar@{..>}[d] \\ N \ar@{..>}[r]|-{\text{\Large \textbullet}} & L_\gen N
} \end{gathered}
 \right]$$
 Then
 $$\Gamma_{\qce} X = \Gamma_{\mathrm{e}} X =  \left[
\begin{gathered} \xymatrix{
& \textcolor{red}{V} \ar@{..>}@[red][d]|-{\textcolor{red}{\text{\Large \textbullet}}}  \\ \textcolor{red}{N'} \ar@{..>}@[red][r] \ar[d] & \textcolor{red}{j_\ast V} \ar[d] \\ N \ar@{..>}[r]|-{\text{\Large \textbullet}} & L_\gen N
} \end{gathered}
 \right]$$
 We need to show that $L_\gen N' \cong j_\ast V$, but this follows
 from the fact that $L_{\gen}$ is strictly smashing: 
 we  apply $L_\gen$ to the diagram to view it in the category of $L_\gen
 \unit$-modules and note that pullbacks preserve isomorphisms. 
\end{proof}

\subsection{The idempotent extendification functor}
\label{subsec:Gie}

The remaining torsion functor is a little more complicated. 

   \begin{lemma}
\label{lem:Gqcie}
   The inclusion $\cCsepie \lra \cCsep$ has a right adjoint $\Gie: \cCsep
   \lra \cCsepie$, and if we have strict generization,    $\Gie$  restricts to a functor 
   $\Gie: \cCsepqc \lra \cCsepqcie$ right adjoint to the inclusion
   $\cCsepqcie\lra \cCsepqc$.
\end{lemma}

\begin{proof}
We suppose given an object $X$ as above and define 
$$\Gie X=
\left[ \begin{gathered} \xymatrix{
& U \ar@{..>}[d] \\ \{M'_\fm\}_{\fm} \ar@{..>}[r]& j_*U
} \end{gathered} \right]
$$
where the factors $M'_\fm$ are defined by the diagram
$$
\xymatrix{
&&U\ar@{=}[dr] \ar@{..>}[dd]&\\
&&&U\ar[dd]\\
M'\ar[dd]\ar[dr]\ar[rr]^\beta &&j_*U\ar[dr]\ar@{..>}[dd]&\\
&\prod_\fm M_\fm \ar@{=}[dd]\ar[rr]^\alpha && P\ar@{=}[dd]\\
\prod_\fm M'_\fm \ar[dr] \ar[rr]^<<<<<<\gamma&&P'\ar[dr]&\\
&\prod_\fm M_\fm \ar[rr]&&P
}$$
where $M'$ is a pullback of $j_*U$ and $\prod_\fm M_\fm$
 and 
$M'_\fm=e_\fm M'$, and $P'$ is defined as a pushout. 
The pullback and pushout are taken in $\Lambda_\fm \unit$-modules,
with $P'$ local because $L_\gen$ preserves pushouts.  
Because the back left hand vertical is ie, so is the
lower back left hand vertical, so $\Gie X$ is indeed idempotent extended. 
It is clear that $\Gie$ is the identity on idempotent extended
objects, so the unit of the adjunction is the identity. The counit
$i\Gie X \lra X$ is displayed in the diagram above, and the triangular
identities are easily checked.  

Furthermore, if the base change of $\alpha$ is an isomorphism, it
follows that the base change of  $\beta$ is an isomorphism
by and hence that the base change of $\gamma$ is an isomorphism. 
\end{proof}

\section{The skeletal Quillen equivalences} 
\label{sec:skeletonequiv}

Referring to the categories in the Roadmap (Subsection 
\ref{subsec:roadmap}), we first need to consider the existence of model 
structures, and then establish the Quillen equivalences.

\subsection{Cellular skeletons of the adelic model}
By Lemma \ref{lem:Gammaepreservesqc},  $\cCadqce$ is a coreflective
subcategory of the cofibrantly generated model category $\cCad$. To
see that it admits a left-lifted model
structure it suffices to show  the generating cofibrations and acyclic cofibrations
are also $\qce$ objects \cite{coreflective}.

\begin{lemma}\label{lem:adwisequivstrong}
For  a weakly qce object $X$,  the counit $\Gamma_{\qce} X 
\xrightarrow{\simeq} X$ is an equivalence in $R_{\wqce} \cCad$, so 
cellularization with respect to $\wqce$ coincides with that for 
$\qce$. 

For  a weakly qc object $X$,  the counit $\Gamma_{\qc} X 
\xrightarrow{\simeq} X$ is an equivalence in $R_{\wqc} \cCad$, so 
cellularization with respect to $\wqc$ coincides with that for $\qc$. 
\end{lemma} 

\begin{proof}
We will provide a chain of weak equivalences $\Gamma_{\qce} X \xrightarrow{\simeq} \Gamma_{\qc} X \xrightarrow{\simeq} X$. We start by fixing our object of interest
$$
X = \left[ \begin{gathered} \xymatrix{
& V \ar@{..>}[d]|-{\text{\Large $\circ$}} \\ N \ar@{..>}[r]|-{\text{\Large $\circ$}} & Q
} \end{gathered} \right] \in \cCadwqce
$$
First we show that there is a weak equivalence
$$ \Gamma_{\qc}X=\left[
 \begin{gathered}
 \xymatrix{
&\textcolor{red}{V'}  \ar[r] \ar@{}[dr]|<<{\lrcorner}
\ar@{..>}@[red][d]|-{\textcolor{red}{\text{\Large $\circ$}}}
 & V \ar@{..>}[d]|-{\textcolor{black}{\text{\Large $\circ$}}}\\ \textcolor{red}{N}
\ar@{..>}@[red][r]|-{\textcolor{red}{\text{\Large \textbullet}}}
 & \textcolor{red}{L_\gen N} \ar[r]^{\simeq} & Q
}
 \end{gathered}
 \right]\xrightarrow{\simeq} X.$$
There is no change at the  $N$ vertex, and by assumption of $X$ being a weak $\qce$-diagram
the map $L_\gen N \to Q$ is a weak equivalence. Therefore we are left
to show that the map $V' \to V$ is a weak equivalence. However this is
the base change of the weak equivalence $L_{\gen}N \stackrel{\simeq} \lra Q$ and the 
model category $\cC$ is assumed to be proper. 

Now we show that there is a weak equivalence $\Gamma_{\qce}X \to \Gamma_{\qc}X$. The object $\Gamma_{\qce}X$ is as follows
$$
\left[
\begin{gathered} \xymatrix{
& \textcolor{red}{V'}
\ar@{..>}@[red][d]|-{\textcolor{red}{\text{\Large \textbullet}}}  \\ 
\textcolor{red}{N'} \ar@{..>}@[red][r]|-{\textcolor{red}{\text{\Large
      \textbullet}}} \ar@{..>}[d] & \textcolor{red}{j_\ast V'}
\ar@{..>}[d]|-{{\text{\Large $\circ$}}}\\
 N \ar@{..>}[r]|-{\text{\Large \textbullet}} & L_\gen N
} \end{gathered}
 \right]
$$
There is no change at the $V'$ vertex, so we need to check at the other two points.
We must show the maps  $N' \to N$ and $j_\ast V' \to L_\gen N$ are
weak equivalences. As the $\Gamma_{\qc}(X)$ is weakly extended, we
have that $j_\ast V' \to L_\gen N$ is a weak equivalence. It then
follows that $N' \to N$ is also a weak equivalence as we are pulling
back along an isomorphism.
\end{proof}

\begin{cor}[Adelic Cellular Skeleton Theorem]\label{cor:adcellskel}
Let $\cC$ be a 1-dimensional Noetherian model category admitting
strict generization. Then inclusions give Quillen equivalences
$$\cC \simeq_QR_{\wqce} \cCad \simeq_Q R_{\we} \cCadqc \simeq_Q \cCadqce.$$ 
We refer to $\cCadqce$ as the cellular skeleton of the adelic model. 
\end{cor}

\begin{proof} The adelic model is Theorem \ref{equiv1}. 
By Subsection \ref{subsec:Gqc} Lemma \ref{lem:Gammaepreservesqc}, $\cCadqc$ and $\cCadqce$
are coreflective subcategories of  $\cCad$.  By Lemma \ref{lem:adwisequivstrong}
 $\cCadqc$ and $\cCadqce$ have model structures 
left-lifted from $\cCad$. By \cite{coreflective} these are cellular skeletons
\end{proof}

\subsection{Model structures on qc-diagram categories} 
\label{subsec:qcmodels}
We consider model categories in the context of the types of categories appearing in Subsections \ref{subsec:qcdiagram} and \ref{subsec:qcpush}. 
Suppose $\cN, \cV, \cQ$ are cofibrantly
generated model categories. We give the diagram category $\cCal$ the diagram-injective model 
structure, cellularize at $R_\wqc$, and give $\cCalqc$ the left-lifted model structure. 

 Since $hF^r$ is not a left adjoint, it is not clear that the entire category of sections $\cCFal$ admits a 
diagram-injective model structure. The main point of this subsection is to observe that the subcategory $\cCFalqc$  of quasi-coherent objects does admit a model structure 
right-lifted along $F^*$ from $\cCalqc$. 

\begin{lemma}
If  $\cN'$ has cylinders, admits a model structure right-lifted  along $F^r$ from 
$\cN$, and the functor $v$ is faithfully flat, then $\cCFalqc$ admits a model structure right-lifted  along $F^*$ from 
$\cCalqc$. In this model structure cofibrant objects have cofibrant
nubs and fibrant objects have fibrant nubs. 
\end{lemma}

\begin{proof}
For an object $N$ of $\cN$ we consider the cospan $a(N)=(N \lra hN
\lla 0)$. The functor $a$ is left adjoint to evaluation: 
$$\Hom_{\cCalqc}(a(N), X)=\Hom_\cN (N, N_X). $$
For an object $U$ of $\cV$ we consider the cospan, 
$b(U)=(0\lra 0\lla U)$, also left adjoint
$$\Hom_{\cCalqc}(b(U), X)=\Hom_\cV (U, fV_X),  $$
but now to the vertex-fibre: $fV_X$ is the subobject of 
$V_X$ mapping to $\ker (vV_X\lra hN_X)$.

Suppose that $I_\cN$ and $J_\cN$ are the sets of generating
cofibrations and acyclic cofibrations for $\cN$ and similarly for
$\cV$. Take $I =\{ a(N), b(U) \st N \in I_\cN, U\in
I_\cV\}$, and similarly for $J$. We then see that $I-inj$ consists of
morphisms with nub component in $I_\cN-inj$ and vertex components 
with $fV_X$ in $I_\cV-inj$. The adjunctions mean that maps out of maps
in $I$ or $J$ are detected either in the nub or in the vertex-fibre.  

Now consider the requirements of \cite[2.1.19]{HoveyMC}. Conditions 1,
2, and 3 are obvious. Condition 4 is also obvious by the adjunctions since these
conditions also hold at vertex and nub. Similarly for Condition 6.  For Condition 5, $I-inj\subseteq
J-inj$ is clear, and it remains to comment on $I-inj\subseteq
\cW$. Using the objects $a(N)$ we see that a map in $I-inj$ is an
equivalence at the nub. From the strict smashing condition it follows
that it is an equivalence also at the adelic vertex. Finally, using
the objects $b(U)$ we see that it is an equivalence at the 
vertex-fibre. It remains to deduce that it is an equivalence at the
vertex. Since $\cN'$ has cyclinders, this can be done with acyclics. We have a pullback square
$$\xymatrix{
fV\ar[r]\ar[d] &0\ar[d]\\
vV \ar[r] & hN
}$$
Now we know that $hN$ and $fV$ are acyclic, so it follows that $hV$ is
acyclic. Since $v$ is faithfully flat, it follows that $V$ is acyclic as required. 
\end{proof}


Finally, we observe that with this model structure,  if the left adjoint $F: \cN \lra \cN'$ is a 
Quillen equivalence then the left adjoint $F_*: \cCalqc \lra \cCFalqc$ is also 
 a Quillen equivalence. 

Suppose $X$ is a cofibrant object of $\cCal$ and $Y$ is a fibrant 
object of $\cCFal$ we must show that $F_*X \lra Y$ is a weak 
equivalence if and only if its adjunct $X\lra F^*Y$ is a weak 
equivalence. However $F_*$ and $F^*$ do not change the vertex, and at 
the nub they are $FN \lra N'$ and $N \lra F^r N'$: since a cofibrant 
object of $\cCalqc$ has cofibrant nub and a fibrant object in 
$\cCFalqc$ has fibrant nub, the result follows.

\subsection{The skeleton of the separated model}
We have shown in Subsection \ref{subsec:Gqc} and Lemma \ref{lem:Gqcie} that 
$$\cCsep\lra \cCsepqc\lra \cCsepqcie$$
are coreflective inclusions. 

\begin{lemma}\label{lem:sepwisequivstrong}\leavevmode
\begin{enumerate}[align=left]
\item[(i)] For  a $\wqcie$ object $X$,  the counit $\Gamma_{\qcie} X 
\xrightarrow{\simeq} X$ is an equivalence in $R_{\wqcie} \cCsep$, so 
cellularization with respect to $\wqcie$ coincides with that for 
$\qcie$. 

\item[(ii)] For  a $\wie$ $\qc$-object $X$,  the counit $\Gamma_{\ie} X 
\xrightarrow{\simeq} X$ is an equivalence in $R_{\wie} \cCsepqc$, so 
cellularization with respect to $\wie$ coincides with that for 
$\ie$. 

\item[(iii)] For  a weakly qc object $X$,  the counit $\Gamma_{\qc} X 
\xrightarrow{\simeq} X$ is an equivalence in $R_{\wqc} \cCsep$, so 
cellularization with respect to $\wqc$ coincides with that for $\qc$. 
\end{enumerate}
\end{lemma} 

\begin{proof}
The proof of Part (iii) is exactly like the corresponding part in the
adelic case. 

For Part (i) we consider the last diagram of Lemma \ref{lem:Gqcie}, and
repeatedly use left and right properness. The notable point is that
the map $M'\lra \prod_\fm M_\fm$ is only an idempotent weak
equivalence, but of course this
implies that $\prod_\fm M_\fm'\lra \prod_\fm M_\fm$ is an actual weak
equivalence. 

The argument for Part (ii) is similar. 
\end{proof}

\begin{thm}
\label{thm:sepskeleton}
(Separated Skeleton Theorem)
Let $\cC$ be a 1-dimensional Noetherian model category admitting
strict generization, and suppose $v$ is faithfully flat.  Then there are Quillen equivalences
$$\cC \simeq_Q L_\Pi  R_\wqcie\cCad =L_\Pi R_\wie R_\wqc\cCad  
\simeq_Q L_\Pi R_\wie \cCadqc  
\simeq_Q R_\wie \cCsepqc  
 \simeq_Q \cCsepqcie$$ 
and hence the skeleton of the separated model $\cCsepqcie$ is Quillen
equivalent to  $\cC$.
\end{thm}

\begin{proof}
By Theorem \ref{thm:adsepcomp}, $\cC$ is equivalent to the separated model $L_\Pi
R_\wqcie\cCad$. 
The equivalence  $L_\Pi R_\wie R_\wqc\cCad  \simeq L_\Pi R_\wie
\cCadqc$  follows from the fact that $\cCadqc$ is a cellular skeleton
of $R_\wqc\cCad$ (Corollary \ref{cor:adcellskel}). The equivalence $ R_\wie \cCsepqc  \simeq
\cCsepqcie$ follows from Lemma \ref{lem:sepwisequivstrong} since the objects of $\cGad$ are in the image of the
reflective inclusion $\sigma^*$. 

The equivalence $L_\Pi R_\wie \cCadqc \simeq L_\Pi R_\wie \cCsepqc$ is
based on the adjoints from Subsection \ref{subsec:sepcospan}.
Subsection \ref{subsec:qcmodels} shows that $\cCsepqc$ admits a model
structure  right-lifted from $\cCadqc$. Accordingly, the $\sigma_*\dashv \sigma^*$ adjunction is
Quillen for the model structures on $\cCadqc$ and  
$\cCsepqc$.  Cellularizing both sides, it remains
Quillen for the model structures on $R_\we \cCadqc$ and 
$R_\wie\cCsepqc$. Finally, it remains a Quillen adjunction relating $L_\Pi
R_\we \cCadqc$ and $R_\wie \cCsepqc$, and this adjunction is a
Quillen equivalence from the definition. 
 \end{proof}

\begin{remark}
We  may consider the following  `super-separated' diagram 
$$
\cC_{\mathrm{ssep}} = \bbGamma \left( \begin{gathered}\xymatrix{
& (L_\gen \unit) \modules_\cC \ar[d]^{\{e_\fm j_\ast\}} \\
\prod_\fm \left[ \Lambda_\fm \unit \modules_\cC \right]
\ar[r]_-{\prod_\fm L_\gen} & \prod_\fm \left[ L_\gen \Lambda_\fm \unit \modules_\cC \right] 
}  \end{gathered} \right)
$$
where the horizontal functor is localization $L_{\gen}$ on each
factor, and the vertical functor is extension of scalars on each
factor. However, this cannot form a model for $\cC$ as we have simply removed too much
information: the splicing data  for assembling the contributions at
individual closed points needs to be `the same almost everywhere' in
the sense that an element of the adeles is given by $(x_p)\in \prod_p
\Q_p$  with $x_p\in \Z_p^{\wedge}$ for almost all $p$.

From another point of view the super-separated square 
fails to be a model because the square 
$$\begin{gathered} \xymatrix{
X \ar[r] \ar[d]  & L_{\gen} \unit \tensor X \ar[d] \\
\displaystyle{\prod_{\fm}} \left[\Lambda_{\fm} \unit \otimes X \right]\ar[r] &
\displaystyle{\prod_{\fm}}\left[ L_{\gen}\Lambda_{\fm} \unit \otimes X \right] \rlap{.}
} \end{gathered}$$
is usually not a homotopy pullback square. 
\end{remark}

\subsection{The  skeleton of the complete model}\label{sec:compskeleton}

Under somewhat more restrictive hypotheses, we identify a local skeleton of the complete model $L_\Lambda R_\wqc \cCad$. 

By the same proof as  Lemma \ref{lem:adwisequivstrong} we compare weak 
and strong $\qc$ objects. 
\begin{lemma}\label{lem:Lkwisequivstrong}
 For  a weakly qc object $X$ in $\cCLk$, the counit $\Gamma_{\qc} X 
\xrightarrow{\simeq} X$ is an equivalence in $R_{\wqc} \cCLk$, so 
cellularization with respect to $\wqc$ coincides with that for $\qc$. \qqed
\end{lemma}

\begin{thm}\label{thm:compskeleton}
Let $\cC$ be a 1-dimensional Noetherian model category which is
formally algebraic, has strict generization and so that $v$ is
faithfully flat. Then there are Quillen equivalences
$$\cC \simeq_Q R_\wqce \cCad \simeq_Q L_\Lambda L_\Pi R_\wqc \cCad \simeq_Q 
 L_\Lambda L_\Pi \cCadqc \simeq_Q 
 L_\Lambda   \cCsepqc \simeq_Q \cCLkqc. $$
and hence the skeleton of the complete model $\cCLkqc$ is Quillen
equivalent to $\cC$.
\end{thm}

\begin{proof}
The equivalences up to $L_\Lambda \cCsepqc$ are exactly as in Theorem
\ref{thm:sepskeleton}.  
For the remaining equivalence $L_\Lambda \cCsepqc \simeq \cCLkqc$, 
we combine the left adjoints $\Lfm : \Rm \modules \lra
\Lfm-\Rm\modules$, to obtain 
$$\ell : \prod_\fm (\Rm \modules )\lra \prod_\fm (\Lfm -\Rm\modules)$$
and then use the  construction of Section 
\ref{subsec:qcpush} to give a functor $\ell_* \colon \cCsepqc \to \cCLkqc$
left adjoint to the inclusion. We give $\cCLkqc$ the right-lifted model
structure as in Subsection \ref{subsec:qcmodels}, giving a Quillen
adjunction relating $\cCsepqc$ and $\cCLkqc$. This localizes further
to give a  Quillen
adjunction relating $L_\Lambda \cCsepqc$ and $\cCLkqc$, and this is a
Quillen equivalence by applying Proposition \ref{prop:Lkmodel} at the nub. 
\end{proof}

\part{Examples}\label{sec:examples}

In the final part of this paper where we make our results explicit in  several
examples, aiming to give a systematic and connected account easing comparison.

\section{Abelian groups}\label{sec:abgrp}

The prototypical example is the derived category of abelian groups
$\cCb =\sfD (\bZ)$.   This applies without  change
if $\bZ$ is replaced by an arbitrary 1-dimensional commutative
Noetherian domain. 

 The underlying category is $\bZ\modules$. Here  $\Z$ is viewed
 as a dg ring entirely in degree 0 with zero differential. Thus  $\bZ\modules$ means
 the category of dg $\bZ$-modules (i.e., chain complexes of abelian
 groups). (When we need to consider abelian groups without grading or
 differential, we say so explicitly). 
 
We make $\bZ\modules$ into a model category using
 the (algebraically)  projective model structure. 
There is a natural order-reversing homeomorphism
$$\spec (\bZ)\stackrel{\cong}\lra \spco (\sfD (\bZ))$$
where the classical {\em algebraic} prime $\fp$ corresponds to the
{\em Balmer} prime $\fp_b=\{ M \st M_\fp\simeq 0\}$ \cite{BalmerSpc}. The closed points are the algebraically maximal primes $(p)$,  and there is a unique generic point corresponding to the algebraic prime $(0)$, and we write $\bQ=\bZ_{(0)}$ for the field of fractions. 

The adelic approximation theorem gives a homotopy pullback square in $\bZ\modules$:
$$\begin{gathered} \xymatrix{
\bZ \ar[r] \ar[d] & \bQ \ar[d] \\
\displaystyle{\prod_{p}} \bZ_p^\wedge \ar[r] & \bQ \tensor
\displaystyle{\prod_{p}} \bZ_p^\wedge  \rlap{ .}} \end{gathered}$$
We note that this square is also a strict pullback square in the
underlying category of abelian groups, familiar as the classical Hasse square. 

In this setting, we have both strict generization and $p$-adic
completion functors, so we will get three different  skeletons.

By Corollary \ref{cor:adcellskel}, the cellular skeleton of the  adelic model is
$$
(\bZ\modules)_{\mathrm{ad}}^{\qce} =
 \bbGamma \left( \begin{gathered}\xymatrix{
& \bQ\modules\ar[d]|{\text{\Large \textbullet}}^{j_\ast} \\ 
(\prod_p \bZ_p^\wedge)\modules \ar[r]|-{\text{\Large \textbullet}}_-{\otimes \bQ} &(\bQ \tensor \prod_p \bZ_p^\wedge) \modules
}  \end{gathered} \right)
$$
The objects of the cellular skeleton of the adelic model  are chain
complexes of diagrams 
$$ \left[ \begin{gathered} 
\xymatrix{
& U \ar@{..>}[d] |{\text{\Large \textbullet }} \\ 
M \ar@{..>}[r]|-{\text{\Large \textbullet}} & P 
} 
\end{gathered} \right]
$$
where $U$ is a $\bQ$-vector space, $M$ is a module over 
$\prod_p \bZ_p^{\wedge}$ and the information at $P$ is an isomorphism of $\bQ \otimes
\prod_p \Z_p^\wedge$-modules $\bQ \otimes M \cong j_\ast U$.

By Theorem \ref{thm:sepskeleton} the  separated skeleton model is
$$
(\bZ\modules)_{\mathrm{sep}}^{\qcie} =\bbGamma 
\left( \begin{gathered}\xymatrix{
& \bQ\modules\ar[d]|{\blacksquare}^-{j_\ast}
\\ 
 \prod_p (\bZ_p^\wedge \modules)\ar[r]|-{\text{\Large \textbullet}}_-{\otimes \bQ \circ \Pi} &(\bQ \tensor \prod_p \bZ_p^\wedge)\modules 
}  \end{gathered} \right) 
$$
where we have used $\blacksquare$ to indicate the ie
condition. The objects of the separated model  are chain complexes of diagrams
$$ \left[ \begin{gathered} 
\xymatrix{
& U \ar@{..>}[d]|-{\blacksquare} \\ 
\{M_p\} \ar@{..>}[r]|-{\text{\Large \textbullet}} & P 
} 
\end{gathered} \right]
$$
where $U$ is a $\bQ$-vector space, each $M_p$ is a module over 
$\bZ_p^{\wedge}$, and  the information at $P$ is given by a  map
$U\lra \bQ\tensor \prod_pM_p$ 
which is extension of scalars along $\Q \lra \Q_p^\wedge$ on each idempotent piece. 

Finally, in the language of \Cref{sec:compskeleton}, the  skeleton of the complete model is 
$$
(\bZ\modules)_{L \kappa}^{\qc} =\bbGamma 
\left( \begin{gathered}\xymatrix{
&\bQ\modules\ar[d]^{j_\ast}
\\ 
 \prod_p  L_0^p-\bZ_p^\wedge\modules \ar[r]|-{\text{\Large \textbullet}}_-{\otimes \bQ \circ \Pi} &(\bQ \tensor \prod_p \bZ_p^\wedge)\modules 
}  \end{gathered} \right)
$$
The objects of the complete model  are chain complexes of diagrams
$$ \left[ \begin{gathered} 
\xymatrix{
& U \ar@{..>}[d] \\ 
\{ M_p\} \ar@{..>}[r]|-{\text{\Large \textbullet}} & P 
} 
\end{gathered} \right]
$$
where $U$ is a $\Q$-vector space,  $M_p$ is an $L_0^p$-complete
$\Z_p^\wedge$-module, and the
information at $P$ is given by 
 an unconstrained map $j_\ast U \to \bQ \otimes \prod_p L^p_0 M_p$ of $\Q \tensor
 \prod_p \bZ_p^\wedge$-modules. 

\begin{remark}
To give a feeling for 
the different models, we  recommend that the 
reader identifies the images in the  adelic,  separated
and  complete models of 
$\sfD (\Z)$ for the abelian groups $\Z, \Q, \Q/\Z, 
\bigoplus_p\Z/p, \prod_p \Z/p$ and moreover explores moving between
the various models. For example the abelian group $\Q$ corresponds to 
the diagrams 
$$\xymatrix{
\Q\ar[r]\ar[d]&\Q\ar[d]     &&\Q\ar[r]\ar[d]&\Q\ar[d]     &&\Q\ar[r]\ar[d]&\Q\ar[d]    \\
\Q\tensor \prod_p\Z_p^{\wedge}\ar[r]& \Q\tensor \prod_p\Z_p^{\wedge}&&
\{ \Q_p\}_p \ar[r] &\prod_p\Q_p&&0\ar[r]& 0 
}$$
in the adelic, separated and complete modules respectively. Similarly 
the abelian group $\Q/\Z$ corresponds to 
the diagrams 
$$\xymatrix{
\Q/\Z\ar[r]\ar[d]&0\ar[d]     &&\Q/\Z\ar[r]\ar[d]&0\ar[d]     &&\Q/\Z\ar[r]\ar[d]&0\ar[d]    \\
\bigoplus_p \Z/p^{\infty}\ar[r]& 0&&
\{ \Z/p^{\infty}\}_p \ar[r] &\Q\tensor \prod_p\Z/p^{\infty}&&\{ \Sigma
\Z_p^{\wedge}\}_p\ar[r]& \Q\tensor\prod_p\Sigma \Z_p^{\wedge} 
}$$
in the adelic, separated and complete modules respectively.  
\end{remark}

\begin{remark}
In this case all of the rings appearing are commutative, and the equivalence is monoidal, so we may wonder about invertible modules. Over $\bZ$, invertible  modules are suspensions of $\bZ$, but when $\bZ$ is
replaced by a number ring $R$, this is more interesting: 
if $M$ is an invertible module $M_{(0)}\cong R_{(0)}$,
and $M_{\fp}\cong R_\fp^{\wedge}$ so
that the module is specified by non-zero elements $x_\fp\in R_{(0)} \tensor
R_{\fp}^{\wedge}$ for each closed point $\fp$. The module is isomorphic to $R$ if
$(x_\fp)$ comes from $R_{(0)}$, so invertible modules  correspond to 
$[R_{(0)}\tensor \prod_\fp R_\fp^{\wedge}]^\times/R_{(0)}^{\times}$ in
a familiar way. 
\end{remark}

\section{Quasi-coherent sheaves on a curve}
\label{sec:qcsheaves}

Let $C$ be an irreducible curve over an algebraically closed field $k$ with
structure sheaf $\cO=\cO_C$, and ring of rational functions $\cK$. We
take $\cC=\Qc (C)$ to be the category of (complexes of) sheaves of
quasi-coherent modules over $C$. The
Balmer spectrum consists of the points of $C$, with a single generic
point containing all the closed points. 
The Adelic Approximation Theorem states that we have a homotopy pullback 
of the diagram 
$$\left( \begin{gathered}\xymatrix{
\cO \ar[r]\ar[d]& \cK\ar[d]\\ 
 \prod_x\cO_x^{\wedge}\ar[r] & \cK\tensor \prod_x \cO_x^{\wedge} 
}  \end{gathered} \right) 
$$

The category $\Qc(C)$ of (chain complexes of)
quasi-coherent sheaves of modules has strict generization, so we may
form the  skeleton of the separated model for  $\Qc(C)$. This
is the category
$$
\Qc(C)_{\mathrm{sep}}^{\qcie} =\bbGamma
\left( \begin{gathered}\xymatrix{
& \cK\modules \ar[d]|{\blacksquare}^{j_\ast} \\ 
 \prod_x \cO_x^{\wedge}\modules  
\ar[r]|-{\text{\Large \textbullet}}_-{\cK \otimes} & \cK\tensor \prod_x \cO_x^{\wedge}\modules  
}  \end{gathered} \right) 
$$
where the product is over closed points $x$. We do not make the adelic
or complete models explicit.

Finally, we note the connection between the Adelic Approximation
Theorem and cohomology of line bundles. 
Tensoring the pullback square for $\cO$ with the line bundle $\cO (D)$ of a divisor $D=\sum_x n_x
(x)$ and taking global sections, there is an exact sequence
$$\xymatrix{
0\ar[r]& H^0(C; \cO (D))\ar[r]&
\cK \times \prod_x\cO_x^{\wedge}\ar[r]^{\langle i, t_x^{-n_x}\rangle}& 
\cK \tensor \prod_x \cO_x^{\wedge}\ar[r]& H^1(C; \cO (D))\ar[r]& 0
}$$
where $t_x$ is a function vanishing to first order at $x$. This
recovers Weil's  residue approach to calculating cohomology as
explained in \cite[II.5]{Serre}

Finally, we note that the Adelic Approximation  Square embodies the
comparison between the Cousin and residue complexes: the top row
gives the Cousin complex
$$\unit \lra L_{\gen}\unit \lra [L_{\gen}\unit] /\unit$$
(which would be $\Z \lra \Q \lra \Q/\Z$ for abelian groups)
and the Mayer-Vietoris sequence of the Hasse-Tate square is Weil's 
residue complex as above (it would be $\Z \lra \Q\oplus \prod_p\Z_p^{\wedge}\lra \Q\tensor
\prod_p\Z_p^{\wedge}$ for abelian groups).

\section{Stable modules for the Klein 4-group}

Our next example comes from modular representation theory.  Let
$G$ be a finite group, and $k$ a field with characteristic
 dividing the order of $G$. We consider the category of
 $kG$-modules. Say that two morphisms $f,g \colon M \to N$ are {\em
   stably homotopic} if $f-g \colon M \to N$ factors through a
 projective  module. The category of $kG$-modules has a model
 structure in which  the weak equivalences are the stable
 equivalences, the fibrations are the epimorphisms and the
 cofibrations are the monomorphisms~\cite[Theorem 2.2.12]{HoveyMC}. The homotopy category of this
 model category is the (big) stable module category
 $\underline{\text{Mod}}kG$.  Passage to syzygies  gives the
 desuspension functor in the stable module category. Tensor product
 over $k$ equips the stable module category with the structure of a
 rigidly small-generated tensor-triangulated category with monoidal unit  $k$.
 
The Balmer spectrum of the small objects is the projective scheme associated to the group
cohomology ring~\cite{BensonCarlsonRickard}:
$$\spco (\underline{\text{Mod}}k G) \cong \operatorname{Proj}
H^\bullet (G;k).$$
In particular, since $H^\bullet(G;k)$ is a Noetherian graded ring by
Venkov's Theorem, the Balmer spectrum is Noetherian. 

When $G=E$ is the Klein 4-group $E = \langle g,h \mid g^2=h^2=1, gh=hg
\rangle$ and $k$ is an algebraically closed field of
characteristic 2, the cohomology ring $H^*(E;k)=\mathrm{Ext}_{kE}^*(k,k)=k[u,v]$ is easy to deal with and
the representation theory  is very well understood.  We take
$x=g-1$ and $y=h-1$ and note  $kE \cong k[x,y]/(x^2,y^2)$. From this
point of view, the stable module category of $kE$ is the singularity category of
$k[x,y]/(x^2, y^2)$ and  the BGG correspondence \cite{BGG}
states  that the singularity category is equivalent to the bounded derived category of
$\bP^1(k)=\mathrm{Proj}(k[u,v])$, giving another identification of the Balmer spectrum.

 It is helpful to display $kE$ as 
$$
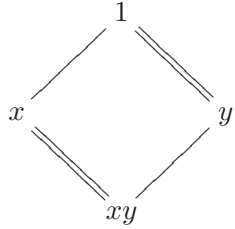
\xymatrix{
& 1 \ar@{-}[dl]  \ar@{=}[dr] &\\
\ar@{=}[dr] x & & y \ar@{-}[dl]  \\
& xy &
}
$$
where a single line refers to multiplication by $x$ and a double line
refers to multiplication by $y$.
One sees syzygies are  as follows
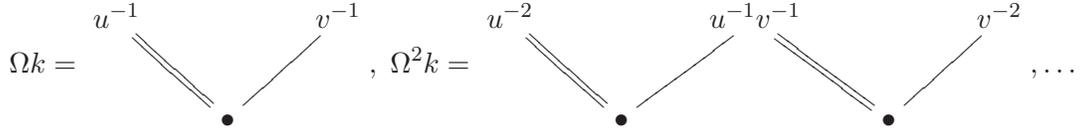
\begin{align*}
 \Omega k =  
\begin{gathered}  \xymatrix{u^{-1} \ar@{=}[dr] & & v^{-1}
     \ar@{-}[dl] \\ & \bullet } \end{gathered}, \;
 \Omega^2 k = 
 \begin{gathered}  \xymatrix{u^{-2} \ar@{=}[dr] & & u^{-1}v^{-1}
     \ar@{-}[dl] \ar@{=}[dr] & & v^{-2} \ar@{-}[dl] \\ & \bullet & &
     \bullet} \end{gathered}, \ldots 
 \end{align*}
The labels refer to the fact that $H^n(E)=\Hom(\Omega^nk, k)$ is the dual of the top
vector space in the diagram: we have chosen the pairing between $k[u,v]$ and
$k[u^{-1},v^{-1}]$ by multiplying and finding the constant term.
Of course $H^*(E)=\mathrm{Symm(E^*)}$, and $\{u,v\}$ is the dual basis to
$\{x,y\}$.  Similarly
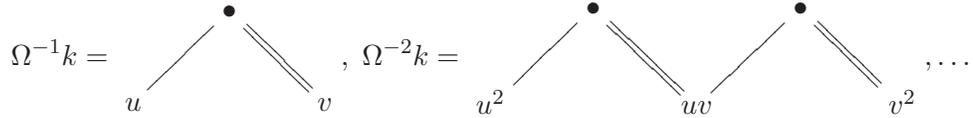
\begin{align*} 
\Omega^{-1}k = \begin{gathered}
\xymatrix{& \bullet  \ar@{-}[dl]
    \ar@{=}[dr] \\ u && v}  \end{gathered}, \; 
\Omega^{-2}k = \begin{gathered}
\xymatrix{& \bullet  \ar@{-}[dl] \ar@{=}[dr] &&\bullet  \ar@{-}[dl] \ar@{=}[dr] &\\
 u^2 && uv&&v^2}  \end{gathered}, \ldots 
\end{align*}
The labels refer to the fact that $H^n(E)=\Hom (k, \Omega^{-n}k)$ is
canonically isomorphic to the bottom row in the diagram. This type of
labelling applies to general modules. In the stable module category, any module is equivalent to one
with no free summands, and for such an $M$, we see that $(xy)=(x,y)^2$
acts as zero.  We therefore write  $V_l=\ann_M (x,y)$ ($l$ for
`lower')  for the socle, and we may choose  a complementary subspace
$V_u$ ($u$ for `upper').  Multiplication by $x,y$ gives
 two maps $x,y: V_u \lra V_l$.  We also have a canonical isomorphism
 $[k,M]_0=V_l$. The group  $[k,M]_*$ has the structure of a module over
 $[k,k]^*=\Hhat^*(E)$, and $V_l$ is its degree 0 part. 
 We will constantly be passing between degree 0
 information and the full graded module, and it is essential to keep
 clear which of the worlds each statement lives in. 
The syzygies of $k$ have the property that  there are endomorphisms 
$t: V_u\lra V_u$ and $t:V_l \lra V_l$ with $y=xt$ and $tx=y$ as
maps $V_u\lra V_l$, and it is natural to think as if $t=v/u$.

We will also need another family of modules. For 
$\zeta \in H^n(G,k) = \Hom_{kG}(\Omega^n k,k)$,  we may take
$\bL_\zeta=\ker (\zeta : \Omega^n k \lra k)$, which only depends on
$\zeta$ up to scalar multiplication. Since $k$ is algebraically
closed, all homogeneous elements of $H^\ast(E,k) = k[u,v]$ are products of 
degree 1 elements, so we need only consider $\bL_{\lambda u+\mu v}$ for
$[\lambda: \mu]\in \bP^1(k)$. The module $k$ has support $\bP^1(k)$ and
for closed points $\zeta$, the module $\bL_{\zeta}$ has support
$\{ \zeta \}$.

Nullifying $\bL_{\zeta}$ has the effect of inverting $\zeta $ in
homotopy. Now $[k,k]_*=\Hhat^*(E;k)$, and the norm sequence of
$H^*(E)$-modules is the short exact sequence
$$0\lra H^*(E)\lra \Hhat^*(E)\lra \Sigma_{-1} H_*(E)\lra 0, $$
where $\Sigma_n$ is the cohomological suspension
$(\Sigma_nM)^c=M^{c-n}$. Since $H_*(E)$ is torsion, 
this shows that inverting $\zeta $ makes $H^*(E)\lra \Hhat^*(E)$ into
an isomorphism. Localizing at the prime $(\zeta)$ involves nullifying
all $\bL_{\psi}$ with $\psi\neq \zeta$ and the generic localization
$L_{\gen}k$ involves nullifying all $\bL_{\zeta}$
 Thus 
$$[k,L_{\zeta}k]_*=k[u,v]_{(\zeta)}, [k, L_{\gen}k]_*=k[u,v]_{(0)}. $$
Note here that $k[u,v]_{(0)}$ is $k(t)$ in degree 0 with $t=v/u$ and
any non-zero element $\zeta \in H^1(E)$ is a periodicity
element. Similarly in $k[u,v]_{(\zeta)}$, except that the periodicity
element is any element not a multiple of $\zeta$.

In fact we can see in both cases how to realize these modules, by
taking the direct limit along multiples of maps $\zeta\in H^1(E)$. It
is convenient to think of $\zeta: k \lra \Omega^{-1}k$ and so forth, so
using the identification of the lower module with the homotopy, we have
$$L_\gen k=  \left( \begin{gathered} \xymatrix{
k(t) \ar@/_/@<-2ex>[d]_x^1 \ar@/^/@<1ex>@{=>}[d]_y^t \\
k(t) 
} \end{gathered} \right)$$
Of course it is natural to think of $k(t)$ the ring of rational functions on $\bP^1(k)$, and similarly, for a closed point $\zeta \in \bP^1(k)$ we can describe $L_\zeta k$ using the rational functions
regular at $\zeta$.
$$L_\zeta  k=  \left( \begin{gathered} \xymatrix{
k[u,v]_{(\zeta)} \ar@/_/@<-2ex>[d]_x^1 \ar@/^/@<1ex>@{=>}[d]_y^t \\
k[u,v]_{(\zeta)}
} \end{gathered} \right)$$

Next, we may calculate the derived completions in general by the formula 
$$\Lambda_{\zeta }M=\holim_s M/\zeta^s. $$ 
Picking $\zeta =u$ for definiteness, one finds
$$[k,\Lambda_{u}k]_*=k[u,v, v^{-1}]_{u}^{\wedge},  $$
where the inversion of $v$ could be replaced by the inversion of any
other degree 1 element (except for multiples of
$u$).   We note that the degree 0 part of this is the completed stalk
of the point $\infty$ (defined by $u=0$): 
$$[k,\Lambda_u k]_0=\left[ \cO_{\bP^1(k)}\right]_{\infty}^{\wedge}. $$

Continuing to the completion we want, we have the usual splitting
$$(L_{\gen}k)/k\simeq \bigoplus_{\zeta}(k[1/\zeta])/k$$
and hence
$$\Lambda_0M:=\Hom(L_{\gen}k/k, M)\simeq \prod_{\zeta} \Lambda_{\zeta}M.$$

The Adelic Approximation Theorem states that  there is a homotopy pullback in the category $\underline{\text{Mod}}(kE)$
$$\xymatrixcolsep{20pt}\xymatrixrowsep{40pt}\xymatrix{
k \ar[r] \ar[d] & L_{\gen}k \ar[d] \\ 
\Lambda_0 k \ar[r] &
L_{\gen} \Lambda_0 k
}$$
A localization of this is considered in relation to the Chromatic
Splitting Conjecture in \cite[Example 3.18]{BHV}.

Applying $[k, \cdot ]_*$ we obtain the square
$$\xymatrixcolsep{20pt}\xymatrixrowsep{40pt}\xymatrix{
\Hhat^*(E) \ar[r] \ar[d] &  k(t)[u,u^{-1}]\ar[d] \\ 
\prod_{\zeta}k[u,v, (\zeta')^{-1}]_{\zeta}^{\wedge}
\ar[r] & 
(\prod_{\zeta}k[u,v, (\zeta')^{-1}]_{\zeta}^{\wedge})_{(0)}
}$$
where $\zeta'$ is any element of $H^1(E)$ not a multiple of $\zeta$.
The associated Mayer-Vietoris sequence of this square is the residue complex for
calculating the cohomology of $\bP^1(k)$ with coefficients in twists
of the structure sheaf as in Section \ref{sec:qcsheaves}.

One can then use the above splitting to build an adelic, separated and
complete model. 

\section{\texorpdfstring{$\mathbb{T}$}{T}-equivariant rational stable
  homotopy theory}
\label{sec:Tspectra}

We finish the examples with a discussion of the case that motivated
this work, namely the category of 
$\T$-equivariant rational cohomology theories, where $\T$ is the
circle group. 

For a general  compact Lie group $G$, we say that an inclusion $K \subseteq H$ of subgroups of $G$ is \emph{cotoral} if
$K$ is normal in $H$ and $H/K$ is a torus. The  Balmer spectrum  consists of the
 conjugacy classes of subgroups under cotoral inclusion~\cite{Greenlees19}. 

In particular, the Balmer spectrum of $\textbf{Sp}_\bQ^{\T}$ is
$$
\spco(\textbf{Sp}^{\T}_\bQ) = 
\begin{gathered}
\xymatrixcolsep{0ex}\xymatrixrowsep{0ex}\xymatrix{&& \T &&\\
 &&&&\\
\cdots &C_4 \ar[uur]&C_3 \ar[uu]&C_2 \ar[uul]&C_1\ar[uull]}
\end{gathered}
$$
where $C_i$ is the cyclic group of order $i$.  The diagram $\bS_{ad}$ coming from the Adelic Approximation Theorem is then
$$\xymatrix{
& L_T \bS \ar[d] \\ \prod_n \Lambda_{C_n} \bS \ar[r] & L_T \prod_n \Lambda_{C_n} \bS \rlap{ .}
}$$
Writing $\cF$ for the family of finite subgroups, this can be written as
$$\bS_{\ad}=\left(\begin{gathered}\xymatrix{& \widetilde{E}\cF \ar[d] \\ \prod_n DE \langle n \rangle_+ \ar[r] & \widetilde{E}\cF \wedge \prod_n DE \langle n \rangle_+ } \end{gathered}\right) \simeq \left(\begin{gathered}\xymatrix{& \widetilde{E}\cF \ar[d] \\ DE\cF_+ \ar[r] & \widetilde{E}\cF \wedge DE\cF_+ } \end{gathered}\right)$$
which is the usual  $\cF$-Tate square for rational $\T$-equivariant
homotopy theory~\cite{Greenlees99} (where $E\langle n \rangle$ is the
mapping cone of $E[\subset C_n]_+\lra E[\subseteq C_n]_+$).  

The adelic model has bifibrant objects
$$
\left(\textbf{Sp}_\bQ^{\T}\right)_{\mathrm{ad}}^{\wqce} = \bbGamma \left( \begin{gathered}\xymatrix{
& (\widetilde{E}\cF)\modules_{\textbf{Sp}_\bQ^{\T}} \ar[d]|{\text{\Large $\circ$}} \\ 
(\prod_n DE \langle n \rangle_+)\modules_{\textbf{Sp}_\bQ^{\T}} \ar[r]|-{\text{\Large $\circ$}} & (\widetilde{E}\cF \wedge \prod_n DE \langle n \rangle_+)\modules_{\textbf{Sp}_\bQ^{\T}}
}  \end{gathered} \right) 
$$
This model is the first step in \cite{GreenleesShipley18} to the algebraic model for rational
$\T$-equivariant spectra. One then shows that taking $\T$-fixed points
termwise gives an equivalence, and the resulting ring spectra are
formal, with homotopy
$$\unitadalg=
\left( \begin{gathered}\xymatrix{
& \bQ\ar[d]\\
\prod_n \bQ [c]
\ar[r] & \cE^{-1}\prod_n 
\bQ [c]
}  \end{gathered} \right) 
$$
We note that there is no single graded ring giving rise to this adelic
diagram, so the use of {\em diagrams} of rings is
essential. Taking modules over this diagram of rings,
we obtain the homotopical version of the standard algebraic model
$$
(\unitadalg \modules)^{\wqce}:=\bbGamma 
\left( \begin{gathered}\xymatrix{
& \bQ \modules \ar[d]|{\text{\Large $\circ$}} \\ 
(\prod_n \bQ [c])\modules
\ar[r]|-{\text{\Large $\circ$}} & (\cE^{-1}\prod_n 
\bQ [c])\modules
}  \end{gathered} \right) 
$$ 
The generic localization is inverting the set $\cE$ generated by the
Euler classes $e(z^n)$ (which is $c$ at subgroups of $C_n$ and 1
otherwise). It is therefore strictly smashing, so  we can form skeletons of the
adelic, separated and complete models.  The proof that the separated
and complete diagrams give models for rational
$\T$-spectra appears here for the first time. The standard model is the 
cellular skeleton of the algebraic model 
$$
dg \cA (\T)= (\unitadalg \modules)^\qce
 :=\bbGamma \left( \begin{gathered}\xymatrix{
& \bQ \modules \ar[d]|{\text{\Large \textbullet}} \\ 
(\prod_n \bQ [c])\modules
\ar[r]|-{\text{\Large \textbullet}} & (\cE^{-1}\prod_n 
\bQ [c])\modules
}  \end{gathered} \right) .
$$ 
This is the category of differential graded objects of the standard 
abelian category $\cA (\T)$ of \cite{Greenlees99}. (The category $\cA
(\T)=\cA_{\ad}(\T)$  is obtained by interpreting `module' in the
classical restricted sense of the abelian category of graded modules
with no differential).   
The  skeleton of the separated model is 
$$
dg \cA_{\sep} (\T)= (\unitadalg \modules)_{\sep}^\qcie 
:=\bbGamma \left( \begin{gathered}\xymatrix{
& \bQ \modules \ar[d]|{\blacksquare}\\ 
\prod_n \left( \bQ [c]\modules\right)
\ar[r]|-{\text{\Large \textbullet}} & (\cE^{-1}\prod_n 
\bQ [c])\modules
}  \end{gathered} \right) .
$$
This is the category of differential graded objects of the separated 
abelian category $\cA_{\sep} (\T)$ defined as for $\cA (\T)$.   
The  skeleton of the complete model is 
$$
dg \cA_{L\kappa} (\T)= (\unitadalg \modules)_{L \kappa}^\qce
:=\bbGamma \left( \begin{gathered}\xymatrix{
& \bQ \modules \ar[d] \\ 
\prod_n \left(L_0^{(c)} -\bQ [c] \modules \right)
\ar[r]|-{\text{\Large \textbullet}} & (\cE^{-1}\prod_n 
\bQ [c])\modules
}  \end{gathered} \right) .
$$
This is the category of differential graded objects of the complete 
abelian category $\cA_{L \kappa} (\T)$ defined as for $\cA (\T)$.   

\addtocontents{toc}{\vspace{5mm}}


\begin{thebibliography}{10}

\bibitem{stratified}
D.~Ayala, A.~Mazel-Gee, and N.~Rozenblum.
\newblock Stratified noncommutative geometry.
\newblock {\em arXiv preprints}, 2019.
\newblock arXiv:1910.14602.

\bibitem{prismatic}
S.~Balchin, T.~Barthel, and J.~P.~C. Greenlees.
\newblock Prismatic decompositions and rational ${G}$-spectra.
\newblock In preparation.

\bibitem{BalchinGreenlees}
S.~Balchin and J.~P.~C. Greenlees.
\newblock Adelic models of tensor-triangulated categories.
\newblock {\em Adv. Math.}, 375:107339, 45, 2020.

\bibitem{adelict}
S.~Balchin, J.~P.~C. Greenlees, L.~Pol, and J.~Williamson.
\newblock Torsion models for tensor-triangulated categories: the
one-step case.
AGT (to appear),  arXiv:2011.10413.

\bibitem{BalmerSpc}
P~Balmer.
\newblock The spectrum of prime ideals in tensor triangulated categories.
\newblock {\em J. Reine Angew. Math.}, 588:149--168, 2005.

\bibitem{BarnesGreenleesKedziorekShipley17}
D.~Barnes, J.~P.~C. Greenlees, M.~K\c{e}dziorek, and B.~Shipley.
\newblock Rational {${\textrm SO}(2)$}-equivariant spectra.
\newblock {\em Algebr. Geom. Topol.}, 17(2):983--1020, 2017.

\bibitem{BarthelFrankland}
T.~Barthel and M.~Frankland.
\newblock Completed power operations for {M}orava {$E$}-theory.
\newblock {\em Algebr. Geom. Topol.}, 15(4):2065--2131, 2015.

\bibitem{BHV}
T.~Barthel, D.~Heard, and G.~Valenzuela.
\newblock The algebraic chromatic splitting conjecture for {N}oetherian ring
  spectra.
\newblock {\em Math. Z.}, 290(3-4):1359--1375, 2018.

\bibitem{Barthel}
T.~Barthel, D.~Heard, and G.~Valenzuela.
\newblock Local duality in algebra and topology.
\newblock {\em Adv. Math.}, 335:563--663, 2018.

\bibitem{Barwick10}
C.~Barwick.
\newblock On left and right model categories and left and right {B}ousfield
  localizations.
\newblock {\em Homology Homotopy Appl.}, 12(2):245--320, 2010.

\bibitem{BGG}
  I. N. Bernsteın, I. M. Gelfand, and S. I. Gelfand,
  ``Algebraic vector bundles on $P^n$ and problems of linear
  algebra'',
  Funct. Anal. Appl. 12 (1978), no. 3, 212–214.

\bibitem{BensonCarlsonRickard}
D.~J. {Benson}, J.~F. {Carlson}, and J.~{Rickard}.
\newblock {Thick subcategories of the stable module category.}
\newblock {\em {Fundam. Math.}}, 153(1):59--80, 1997.

\bibitem{Bergner12}
J.~E. Bergner.
\newblock Homotopy limits of model categories and more general homotopy
  theories.
\newblock {\em Bull. Lond. Math. Soc.}, 44(2):311--322, 2012.

\bibitem{Greenlees99}
J.~P.~C. Greenlees.
\newblock Rational {$S^1$}-equivariant stable homotopy theory.
\newblock {\em Mem. Amer. Math. Soc.}, 138(661):xii+289, 1999.

\bibitem{Greenlees01b}
J.~P.~C. Greenlees.
\newblock Tate cohomology in axiomatic stable homotopy theory.
\newblock In {\em Cohomological methods in homotopy theory ({B}ellaterra,
  1998)}, volume 196 of {\em Progr. Math.}, pages 149--176. Birkh\"{a}user,
  Basel, 2001.

\bibitem{Greenlees19}
J.~P.~C. Greenlees.
\newblock The {B}almer spectrum of rational equivariant cohomology theories.
\newblock {\em J. Pure Appl. Algebra}, 223(7):2845--2871, 2019.

\bibitem{GreenleesMay92}
J.~P.~C. Greenlees and J.~P. May.
\newblock Derived functors of {$I$}-adic completion and local homology.
\newblock {\em J. Algebra}, 149(2):438--453, 1992.

\bibitem{GreenleesMay95}
J.~P.~C. Greenlees and J.~P. May.
\newblock Generalized {T}ate cohomology.
\newblock {\em Mem. Amer. Math. Soc.}, 113(543):viii+178, 1995.

\bibitem{GreenleesShipley13}
J.~P.~C. Greenlees and B.~Shipley.
\newblock The cellularization principle for {Q}uillen adjunctions.
\newblock {\em Homology Homotopy Appl.}, 15(2):173--184, 2013.

\bibitem{GreenleesShipley14b}
J.~P.~C. Greenlees and B.~Shipley.
\newblock Homotopy theory of modules over diagrams of rings.
\newblock {\em Proc. Amer. Math. Soc. Ser. B}, 1:89--104, 2014.

\bibitem{GreenleesShipley18}
J.~P.~C. Greenlees and B.~Shipley.
\newblock An algebraic model for rational torus-equivariant spectra.
\newblock {\em J. Topol.}, 11(3):666--719, 2018.

\bibitem{coreflective}
T.~Haraguchi.
\newblock On model structure for coreflective subcategories of a model
  category.
\newblock {\em Math. J. Okayama Univ.}, 57:79--84, 2015.

\bibitem{Hirschhorn03}
P.~S. Hirschhorn.
\newblock {\em Model categories and their localizations}, volume~99 of {\em
  Mathematical Surveys and Monographs}.
\newblock American Mathematical Society, Providence, RI, 2003.

\bibitem{HoveyMC}
M.~Hovey.
\newblock {\em Model categories}, volume~63 of {\em Mathematical Surveys and
  Monographs}.
\newblock American Mathematical Society, Providence, RI, 1999.

\bibitem{HoveyStrickland99}
M.~Hovey and N.~P. Strickland.
\newblock Morava {$K$}-theories and localisation.
\newblock {\em Mem. Amer. Math. Soc.}, 139(666):viii+100, 1999.

\bibitem{HPS}
M.~Hovey, J.~H.~Palmieri and N.~P. Strickland.
\newblock Axiomatic stable homotopy theory.
\newblock {\em \em Mem. Amer. Math. Soc.}, 128(610):x+114, 1997

\bibitem{HA}
J.~{Lurie}.
\newblock {\em {Higher algebra}}.
\newblock 2017.
\newblock Avaliable from the author's webpage at
  \url{http://www.math.harvard.edu/~lurie/papers/HA.pdf}.

\bibitem{PolWilliamson}
L.~Pol and J.~Williamson.
\newblock The {L}eft {L}ocalization {P}rinciple, completions, and cofree
  {$G$}-spectra.
\newblock {\em J. Pure Appl. Algebra}, 224(11):106408, 33, 2020.

\bibitem{SchwedeShipley00}
S.~Schwede and B.~Shipley.
\newblock Algebras and modules in monoidal model categories.
\newblock {\em Proc. London Math. Soc. (3)}, 80(2):491--511, 2000.

\bibitem{Serre}
  J.-P. Serre,
  ``Algebraic groups and class fields''
  Graduate Texts in Mathematics,   {\bf 117}, 1988, Springer-Verlag,
  New York, pp x+207,
    

\end{thebibliography}
\end{document}